\documentclass{amsart}

\usepackage[numbers]{natbib}		
\usepackage{fancyhdr}				
\usepackage{amssymb} 				
\usepackage{amsthm} 				
\usepackage{thmtools} 				
\usepackage{mathtools}			
\usepackage[dvipsnames]{xcolor}	
\usepackage{stmaryrd} 			
\usepackage{tikz}
\usepackage{bm}
\usepackage{cancel}

\tikzstyle{every node} = [draw, fill=white, circle, inner sep=0pt, minimum size=5pt]
\tikzstyle{n} = [draw=none, rectangle, inner sep=0pt] 
\tikzstyle{i} = [draw, fill=black, circle, inner sep=0pt, minimum size=5pt] 
\tikzstyle{e} = [draw=none, rectangle, inner sep=0pt]
\tikzstyle{nci} = [draw, fill=black, rectangle, inner sep=0pt, minimum size=5pt]
\tikzstyle{nc} = [draw, fill=white, rectangle, inner sep=0pt, minimum size=5pt]
\tikzstyle{boolean} = [line width=2pt]

\declaretheorem[style=plain,numberwithin=section]{theorem}

\declaretheorem[style=plain,sibling=theorem]{proposition}

\declaretheorem[style=plain,sibling=theorem]{lemma}
\declaretheorem[style=plain,sibling=theorem]{corollary}
\declaretheorem[style=definition,sibling=theorem]{example}

\newcommand{\alg}[1]{\mathbf{#1}}		
\newcommand{\rd}{\slash}			
\newcommand{\ld}{\backslash}			
\newcommand{\negsim}{{\sim}}			
\newcommand{\negmin}{{-}}			
\newcommand{\CIdInRL}{\mathsf{CIdInRL}}	
\newcommand{\mleq}{\sqsubseteq}		

\DeclareSymbolFont{symbolsC}{U}{txsyc}{m}{n}
\SetSymbolFont{symbolsC}{bold}{U}{txsyc}{bx}{n}
\DeclareFontSubstitution{U}{txsyc}{m}{n}
\DeclareMathSymbol{\diamonddot}{\mathord}{symbolsC}{144}

\begin{document}

\title[Commutative idempotent involutive residuated lattices]{The structure of finite commutative idempotent involutive residuated lattices}

\author{Peter Jipsen}
\address{Faculty of Mathematics, Keck Center of Science and Engineering, Chapman University, 1 University Drive, Orange, CA, 92866, USA}
\email{jipsen@chapman.edu}

\author{Olim Tuyt}
\address{Mathematical Institute, University of Bern, Sidlerstrasse 5, 3012, Bern, Switzerland}
\email{olim.tuyt@math.unibe.ch}

\author{Diego Valota}
\address{Dipartimento di Informatica, Universit\`{a} degli Studi di Milano,Via Celoria 18, 20133 Milano, Italy}
\email{valota@di.unimi.it}


\begin{abstract} 
We characterize commutative idempotent involutive residuated lattices as 
disjoint unions of Boolean algebras arranged over a distributive lattice.
We use this description to introduce a new construction, called gluing, that allows us to build new members of
this variety from other ones. In particular, all finite members can be constructed in this way
from Boolean algebras.
Finally, we apply our construction to prove that the fusion reduct of any finite member is a distributive semilattice, 
and to show that this variety is not locally finite.
\end{abstract}

\maketitle

\section{Introduction}\label{sec:intro}

Residuated lattices are algebraic structures that provide an algebraic semantics for substructural logics \cite{GJKO07}. The variety $\mathsf{RL}$ of residuate lattices originated in the 1930s \cite{WD39}, and includes well-known classes of algebras such as Heyting algebras and lattice-ordered groups.

In this paper we study the variety $\CIdInRL$ of commutative idempotent involutive residuated lattices. From a logical perspective, commutative involutive residuated lattices are the algebraic semantics of multiplicative additive linear logic (\textsf{MALL}). Notable subvarieties of $\CIdInRL$ include the classes of Boolean algebras and Sugihara monoids, introduced as an algebraic semantics of the relevance logic $\textsf{RM}^t$ \cite{Du70}. 
Some structural results
about idempotent residuated lattices that are not necessarily commutative or involutive can be
found in \cite{R07} and \cite{FJM19}.
We obtain a structural characterization for all the finite members of $\CIdInRL$, or equivalently, all finite idempotent \textsf{MALL}-algebras. We reach our goal by describing the members of $\CIdInRL$ as disjoint unions of Boolean algebras under the distributive lattice order given by the commutative idempotent monoidal operation of our residuated lattices, with involution as complementation within each Boolean algebra.
We give a procedure to construct from two algebras in $\CIdInRL$ a new member of $\CIdInRL$. Moreover, we show that this procedure can be reversed for the finite members of $\CIdInRL$, giving us the structural description.

The structure of finite Sugihara monoids was already known to J.\ M.\ Dunn, who showed in \cite{ABN75}
that subdirectly irreducible Sugihara monoids are linearly ordered and described the structure
of finite Sugihara monoid chains in \cite{Du70}. Idempotent residuated
chains are given a structural description in \cite{CZ09}, and the structure of conic idempotent residuated lattices has been investigated in \cite{CZG09}.

More recently S.\ Jenei \cite{Je20,Je20a} 
gives a full structural description of all even or odd involutive commutative residuated chains.
This is a significant generalization of Dunn's results since idempotence is not assumed, and the idempotent elements in these algebras form a subalgebra that is a Sugihara monoid. Our results are a generalization in a different direction since we do not assume that the algebras are linearly ordered, but we do assume that all elements are idempotent. In \cite{R07} and \cite{Je20,Je20a} some lemmas about the term $|x|=\tau(x)=x\to x$ are proved and we reference these results in Section~\ref{sec:partition}, though we use the (equivalent) notation $1_x=x\vee {-}x$.

The structure of this paper is as follows: in Section~\ref{sec:prelim}, we introduce the basic definitions necessary for the rest of the paper. In Section~\ref{sec:partition}, we show that each algebra in $\CIdInRL$ is a disjoint union of Boolean algebras such that this union forms a distributive lattice. In Section~\ref{sec:filters} we investigate the congruences and filters of algebras in $\CIdInRL$, laying the groundwork for the gluing construction. In Section~\ref{sec:gluing} we outline the gluing construction and in Section~\ref{sec:reverse} we prove that every finite member of $\CIdInRL$ can be obtained as a gluing of smaller members, resulting in the structural characterization. Lastly, in Section~\ref{sec:app} we mention two applications of this characterization. 

\section{Preliminaries}\label{sec:prelim}

In this section we collect basic properties and definitions for the algebraic structures that we need in our investigation. A \emph{(pointed) residuated lattice} is a tuple $\alg{A} = \langle A, \wedge, \vee, \cdot, \ld, \rd, 1, 0 \rangle$ such that $\langle A, \wedge, \vee \rangle$ is a lattice, $\langle A, \cdot, 1 \rangle$ is a monoid, and $\cdot$ is residuated in the underlying lattice order with residuals $\ld$ and $\rd$, i.e. for $x,y,z \in A$,
	\[
	x \cdot y \leq z \qquad \iff \qquad x \leq z \rd y \qquad \iff \qquad y \leq x \ld z.
	\]
We call $\alg{A}$ \emph{idempotent} if $x \cdot x = x$ for all $x \in A$. We say that $\alg{A}$ is \emph{commutative} if $x \cdot y = y \cdot x$ for all $x,y \in A$. In this case, the two residuals coincide, e.g. $x \ld y = y \rd x$ for all $x,y \in A$, and we replace $\ld$ and $\rd$ in the signature with a single implication arrow $x \to y \coloneqq x \ld y$.
The \emph{linear negations} are defined by $\negmin x \coloneqq 0 \rd x$ and $\negsim x \coloneqq x \ld 0$. We call $\alg{A}$ \emph{involutive} if $\negsim \negmin x = x = \negmin \negsim x$ for all $x \in A$, whence it follows that $\negmin 1=\negsim 1=0$.

In this paper, we focus on the class of commutative idempotent involutive residuated lattices, denoted $\CIdInRL$. This class can be equationally defined over that of residuated lattices and hence forms a variety.\footnote{We say that a residuated lattice is \emph{cyclic} if $\negmin x = \negsim x$ for all $x \in A$. Clearly, any commutative residuated lattice is cyclic. It was shown by Jos\'e Gil-F\'erez in unpublished work that the converse also holds in the context of idempotent involutive residuated lattices. So the class $\CIdInRL$ coincides with the class of cyclic idempotent involutive residuated lattices.}

Note that for any pointed residuated lattice $\alg{A}$, we have the following equivalence for any $x,y \in A$:
\[
x \leq \negmin y \qquad \iff \qquad x \cdot y \leq 0 \qquad \iff \qquad y \leq \negsim x.
\]
Moreover, if $\alg{A}$ is involutive the residuals are definable in terms of fusion and involution, namely $x \ld y = \negsim (\negmin y \cdot x)$ and $y \rd x = \negmin (x \cdot \negsim y)$. This suggests that we can replace the residuals in the signature with the negations, which is indeed the case. 
An \textsf{InRL}-\emph{semiring} is an algebra $\alg{A} = \langle A, \vee, \cdot, 1, \negsim, \negmin \rangle$ such that $\langle A,\vee\rangle$ is a semilattice, $\langle A,\cdot,1\rangle$ is a monoid, $\negsim \negmin x = x = \negmin \negsim x$ for all $x \in A$, and for all $x,y \in A$, 
\[x \leq \negmin y\qquad\iff\qquad x \cdot y \leq \negmin 1\qquad\iff\qquad y \leq \negsim x.\]

\begin{theorem}[{\cite[\S 3.3.5]{GJKO07}}]\label{thm:termequiv}
\textsf{InRL}-semirings are term equivalent to involutive residuated lattices.
\end{theorem}
The element $\negmin 1$ is denoted by $0$, and it follows that $0=\negsim 1$.
Since $\cdot$ distributes over $\vee$ from the left and right in all residuated lattices, this also holds for \textsf{InRL}-semirings (as would be expected for semirings; note however that for \textsf{InRL}-semirings $0$ is in general not an absorbing element).
We call an \textsf{InRL}-semiring \emph{idempotent} if $x \cdot x = x$ for all $x \in A$ and \emph{commutative} if $x \cdot y = y \cdot x$ for all $x, y \in A$. 

For a commutative idempotent residuated lattice $\alg{A}$, the fusion operator $\cdot$ defines a semilattice order $\mleq$ where $x \mleq y \vcentcolon\Leftrightarrow x \cdot y = x$. We sometimes refer to the order $\mleq$ as the \emph{monoidal order}. Note that for a commutative idempotent residuated lattice $\alg{A}$, $\mleq$ is a meet-semilattice order with $1$ as its top element. This allows for a graphical representation of any $\alg{A} \in \CIdInRL$. To represent $\alg{A}$, it suffices to draw the Hasse diagrams of the lattice order $\leq$ and monoidal order $\mleq$ and define the involution. The division operators can be derived from this information.

\begin{example}
	Consider the algebra $\alg{A}_1 \in \CIdInRL$ as shown in Figure~\ref{fig:A1}. Note that $\negmin \bot = \top$ and $\negmin 0 = 1$.
\end{example}

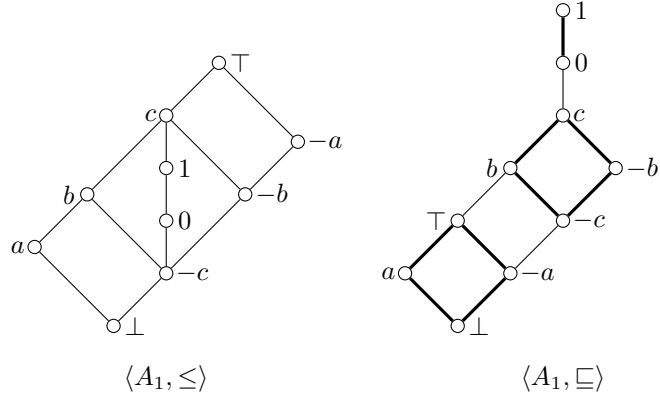
\begin{figure}[h!]
\begin{center}
\begin{tikzpicture}[baseline=0pt,scale=0.7]
\node at (0,-1)[n]{$\langle A_1, \leq \rangle$};
\node at (0,-2)[n]{$ $};
\node(9) at (1,5)[label=right:$\top$]{};
\node(8) at (0,4)[label=left:$c$]{};
\node(7) at (0,3)[label=right:$1$]{};
\node(6) at (0,2)[label=right:$0$]{};
\node(5) at (-1.5,2.5)[label=left:$b$]{};
\node(4) at (2.5,3.5)[label=right:$-a$]{};
\node(3) at (1.5,2.5)[label=right:$-b$]{};
\node(2) at (0,1)[label=right:$-c$]{};
\node(1) at (-2.5,1.5)[label=left:$a$]{};
\node(0) at (-1,0)[label=right:$\bot$]{};
\draw(9) ;
\draw(8)  edge (9);
\draw(7)  edge (8);
\draw(6)  edge (7);
\draw(5)  edge (8);
\draw(4)  edge (9);
\draw(3)  edge (4) edge (8);
\draw(2)  edge (3) edge (5) edge (6);
\draw(1)  edge (5);
\draw(0)  edge (1) edge (2);
\end{tikzpicture}
\quad
\begin{tikzpicture}[baseline=0pt,scale=0.7]
\node at (3,-1)[n]{$\langle A_1, \mleq \rangle$};
\node at (3,-2)[n]{$ $};
\node(9) at (3,6)[label=right:$1$]{};
\node(8) at (3,5)[label=right:$0$]{};
\node(7) at (3,4)[label=right:$c$]{};
\node(6) at (2,3)[label=left:$b$]{};
\node(5) at (4,3)[label=right:$-b$]{};
\node(4) at (3,2)[label=right:${-}c$]{};
\node(3) at (1,2)[label=left:$\top$]{};
\node(2) at (2,1)[label=right:${-}a$]{};
\node(1) at (0,1)[label=left:$a$]{};
\node(0) at (1,0)[label=right:$\bot$]{};
\draw(9) ;
\draw[very thick](8)  edge (9);
\draw(7)  edge (8);
\draw[very thick](6)  edge (7);
\draw[very thick](5)  edge (7);
\draw[very thick](4)  edge (5) edge (6);
\draw(3)  edge (6);
\draw(2)  edge (4);
\draw[very thick](2)  edge (3);
\draw[very thick](1)  edge (3);
\draw[very thick](0)  edge (1) edge (2);
\end{tikzpicture}
\end{center}
\caption{The two orders $\mleq$ and $\leq$ of an algebra $\alg{A}_1 \in \CIdInRL$. 
}\label{fig:A1}
\end{figure}

\begin{example}\label{ex:sugihara}
A well-known subclass of $\CIdInRL$ is the variety $\mathsf{SM}$ of \emph{Sugihara monoids}, introduced as the algebraic semantics for the relevance logic $\textsf{RM}^t$ \cite{ABN75}. Here, a Sugihara monoid is a commutative idempotent involutive residuated lattice for which the underlying lattice order is distributive. In fact, Dunn showed that all Sugihara monoids are also \emph{semilinear} \cite[\S29.4]{ABN75}, i.e., they are the subdirect product of totally ordered members of $\CIdInRL$ or, equivalently, satisfy the equation $((x \to y) \wedge 1) \vee ((y \to x) \wedge 1) = 1$. Note that $\alg{A}_1$ in Figure~\ref{fig:A1} is neither distributive nor does it satisfy the equation for semilinearity (consider $x = a$ and $y = \negmin b$). Hence, $\CIdInRL$ satisfies neither distributivity nor semilinearity and so $\textsf{SM}$ is a proper subclass of $\CIdInRL$.
\end{example}

Given a residuated lattice $\alg{A}$, the set $A^+ \coloneqq \{ x \in A \mid 1 \leq x \}$ is called the \emph{positive cone} of $\mathbf{A}$ and its members are called \emph{positive elements} of $\alg{A}$, whereas the set $A^- \coloneqq \{ x \in A \mid x \leq 1 \}$ is called the \emph{negative cone} of $\mathbf{A}$ and its members \emph{negative elements} of $\alg{A}$. 
The set $A^-$ induces a (not necessarily pointed) residuated lattice $\alg{A}^- \coloneqq \langle A^-, \wedge, \vee, \cdot, \ld^{\alg{A}^-},\rd^{\alg{A}^-}, 1 \rangle$ where for all $a,b \in A^-$, 
\[
a \ld^{\alg{A}^-} b \coloneqq a \ld b \wedge 1 \enspace \text{and} \enspace a \rd^{\alg{A}^-} b \coloneqq a \rd b \wedge 1.
\]
In general, $0$ is not a member of $A^-$. However, if $\alg{A} \in \CIdInRL$, then $0 \in A^-$, as is clear from the following lemma. This lemma summarizes a number of elementary properties of members of $\CIdInRL$ without proof, used throughout the paper without explicit reference to this lemma.

\begin{lemma}
Let $\alg{A} \in \CIdInRL$. Then the following properties hold for all $x,y \in A$:
\[
\begin{array}{rlrl}
(1)	& x \leq y \text{ if and only if } \negmin y \leq \negmin x	& (5)	& \negmin(x \wedge y) = \negmin x \vee \negmin y \\
(2)	& x \wedge y \leq x \cdot y \leq x \vee y & (6)	& \negmin(x \vee y) = \negmin x \wedge \negmin y \\
(3)	& \text{if } x,y \in A^+, \text{ then } x \cdot y = x \vee y	& (7)	& x \to y = \negmin (x \cdot \negmin y) \\
(4)	& \text{if } x,y \in A^-, \text{ then } x \cdot y = x \wedge y	& (8)	& \negmin 1 = 0 \leq 1 = \negmin 0
\end{array}
\]
\end{lemma}

\section{Boolean Partition}\label{sec:partition}

In this section, we show that any $\alg{A} \in \CIdInRL$ can be partitioned into Boolean algebras such that these Boolean algebras form a distributive lattice. For each $x \in A$, we define elements $1_x \coloneqq x \vee \negmin x$ and $0_x \coloneqq x \wedge \negmin x$ and consider the interval $B_x \coloneqq \{ y \in A \mid 0_x \mleq y \mleq 1_x \}$.
Some parts of the following lemmas appeared previously in \cite[Prop. 10]{R07} and \cite[Section 2]{Je20} using the notation $|x|=\tau(x)=x\to x$ for $1_x$. For the sake of completeness we reprove them here. The sets $B_x$
also occur in \cite{Je20} as $X_{\tau = x}$ in Def. 5.1, where this concept is applied to involutive FL$_e$-chains. In that setting $B_x$ has cardinality $\le 2$ whereas it follows from Theorem~\ref{BA} below that in the non-linearly-ordered idempotent setting $B_x$ can be any Boolean algebra.

\begin{lemma}\label{lem:facts01}
Let $\alg{A} \in \CIdInRL$. Then for each $x \in A$,
	\begin{enumerate}
		\item $\negmin 0_x = 1_x$;
		\item $0_x = x \cdot \negmin x$;
		\item $0_x \leq 0 \leq 1 \leq 1_x$ and $0_x \mleq 1_x$;
		\item for $y,z \in B_x$, $y \cdot z = y \wedge z$.
	\end{enumerate}
\end{lemma}
\begin{proof}
(1) By De Morgan laws.

(2) The inequality $0_x \leq x \cdot \negmin x$ follows by idempotence. For the other direction, note that from $\negmin x \leq \negmin x$ we obtain $x \cdot \negmin x \cdot \negmin x = x \cdot \negmin x \leq 0$ and thus $x \cdot \negmin x \leq x$ by residuation. Similarly, $x \cdot \negmin x \leq \negmin x$ and therefore $x \cdot \negmin x \leq 0_x$.

(3) By (2) and residuation, we obtain $0_x = x \cdot \negmin x \leq 0$ and so $1 = \negmin 0 \leq \negmin 0_x = 1_x$ by (1).
Moreover, $0_x \cdot 1_x = 0_x \cdot (x \vee \negmin x) = (0_x \cdot x) \vee (0_x \cdot \negmin x) = 0_x \vee 0_x = 0_x$.

(4) Consider any $y,z \in B_x$. The inequality $y \wedge z \leq y \cdot z$ follows by idempotence. For the other inequality, note that as $0_x \cdot y = 0_x \leq 0$ by (3) and $y \in B_x$, we obtain $y \leq \negmin 0_x = 1_x$ and so $y \cdot z \leq 1_x \cdot z = z$ as $z \in B_x$. Analogously, we get $y \cdot z \leq y$ and so $y \cdot z \leq y \wedge z$.
\end{proof}

Note that by (1) and (2) of the Lemma~\ref{lem:facts01} above, $1_x = -(x \cdot \negmin x) = x \to x$ for each $x \in A$. Moreover, by (4) of the same lemma, inside any interval $B_x$ the lattice and monoidal order coincide. 
That is, for $y, z \in B_x$, $y \mleq z$ if and only if $y \leq z$. The following lemma establishes that each $B_x$ is closed under involution.

\begin{lemma}\label{lem:intervalnegclosure}
Let $\alg{A} \in \CIdInRL$ and $x \in A$.
	\begin{enumerate}
		\item For all $y \in B_x$, $\negmin y = y \to 0_x$.
		\item For all $y \in B_x$, $\negmin y \in B_x$.
		\item For all $y \in B_x$, $0_y = 0_x$.
	\end{enumerate}
\end{lemma}
\begin{proof}
(1) Using Lemma~\ref{lem:facts01}(1), $y \mleq 1_x$ yields that $y \to 0_x = \negmin (y \cdot \negmin 0_x) = \negmin (y \cdot 1_x) = \negmin y$.
		
(2) We first show that $0_x \mleq y \to 0_x$. Note that $0_x \cdot (y \to 0_x) = 0_x \cdot y \cdot (y \to 0_x) \leq 0_x \cdot 0_x = 0_x$. Moreover, $y \cdot 0_x = 0_x \leq 0_x$, i.e. $0_x \leq y \to 0_x$. This gives $0_x = 0_x \cdot 0_x \leq 0_x \cdot (y \to 0_x)$. Hence $0_x \cdot (y \to 0_x) = 0_x$, i.e. $0_x \mleq y \to 0_x$.
		
\noindent Secondly, we show that $y \to 0_x \mleq 1_x$. From $y \cdot (y \to 0_x) \cdot 1_x \leq 0_x \cdot 1_x = 0_x$, we obtain $(y \to 0_x) \cdot 1_x \leq y \to 0_x$. Moreover, $y \to 0_x = (y \to 0_x) \cdot 1 \leq (y \to 0_x) \cdot 1_x$. Therefore, $(y \to 0_x) \cdot 1_x = y \to 0_x$ and hence $y \to 0_x \mleq 1_x$. By (1), we are done.
		
(3) Reasoning as in (1), we have $0_y = \negmin y \cdot y \leq 0_x$. Now, as $y \in B_x$, we have $0_x \cdot y = 0_x \leq 0$ and so $0_x \leq \negmin y$. By (2), we also have $\negmin y \in B_x$ and so $0_x \cdot \negmin y = 0_x \leq 0$, i.e. $0_x \leq y$. Therefore, $0_x \leq y \wedge \negmin y = 0_y$.
\end{proof}

\begin{theorem}\label{BA}
Let $\alg{A} \in \CIdInRL$ and $x \in A$. Then $\langle B_x, \cdot, \vee, {\negmin}, 0_x, 1_x \rangle$ is a Boolean algebra.
\end{theorem}
\begin{proof}
	First we observe that $B_x$ is closed under all the operations. It is closed under ${\negmin}$ by Lemma~\ref{lem:intervalnegclosure}(2). Since $\cdot$ is the meet for $\mleq$, $B_x$ is closed under $\cdot$. Closure under $\vee$ then follows by De Morgan and Lemma~\ref{lem:facts01}. By Lemma~\ref{lem:intervalnegclosure}(1) $\langle B_x, \cdot, \negmin, 0_x \rangle$ is a pseudocomplemented semilattice, i.e. $y \cdot z = 0_x$ if and only if $z \mleq \negmin y$ for any $y,z \in B_x$.	
By \cite{F41}, a pseudocomplemented lattice satisfying $\negmin \negmin x= x$ is a Boolean algebra. 
So indeed, $\langle B_x, \cdot, \vee, {\negmin}, 0_x, 1_x \rangle$ is a Boolean algebra.
\end{proof}

\begin{proposition}
Let $\alg{A} \in \CIdInRL$. Then the collection $\{ B_x \mid x \in A \}$ partitions $A$.
\end{proposition}
\begin{proof}
Note that by Lemma~\ref{lem:intervalnegclosure}(3), we have that for all $x, y \in A$, if $y \in B_x$, then $B_x = B_y$.  We are left to show that $x \in B_x$.  By idempotence, we have that $0_x = x \cdot \negmin x \mleq x$. Moreover, $\negmin x \cdot x \cdot 1_x = 0_x \cdot 1_x = 0_x \leq 0$ implies that $x \cdot 1_x \leq x$ and $1 \leq 1_x$ implies that $x \leq x \cdot 1_x$. Hence $x \mleq 1_x$.
\end{proof}

We have now obtained a partition of $A$ in terms of the intervals $\{ B_x \mid x \in A\}$. We show that the set $\{ 0_x \mid x \leq 0 \}$ forms a distributive lattice with respect to the induced order.

\begin{theorem}\label{distriblattice}
Let $\alg{A} \in \CIdInRL$. Then $\{ 0_x \mid x \in A \} = \{ x \in A \mid x \leq 0 \}$ and $\{ 1_x \mid x \in A \} = A^+$. The algebra $\langle \{ 0_x \mid x \in A \}, \cdot, \vee, 0 \rangle$ is a distributive sublattice of $\alg{A}$ with maximum element $0$, and it is dually isomorphic to $\langle A^+,\wedge, \vee, 1\rangle = \langle A^+, \wedge, \cdot, 1 \rangle$.
\end{theorem}
\begin{proof}
That $\{ 0_x \mid x \in A \} = \{ x \in A \mid x \leq 0 \}$ holds, follows from $0_x \leq 0$  as well as that $x \leq 0$ implies $x \cdot \negmin x = x$ for all $x \in A$. It then also follows that $\{ 1_x \mid x \in A \} = A^+$. 
	
It is easy to see that $\langle \{ x \in A \mid x \leq 0 \}, \wedge, \vee, 0 \rangle$ is a sublattice of $\alg{A}$ with maximum element $0$. Now note that $x \wedge y = x \cdot y$ for all $x,y \in A^-$ and so $\langle \{ x \in A \mid x \leq 0 \}, \cdot, \vee, 0 \rangle$ is indeed a distributive sublattice of $\alg{A}$ with maximum element $0$. The dual isomorphism follows from the De Morgan laws.
\end{proof}

\begin{example}
The algebra $\alg{A}_1$ in Figure~\ref{fig:A1} contains three Boolean algebras whose universes are $B_a=\{\bot,\top, a, -a\}$, $B_b=\{b,-b,c,-c\}$ and $B_0=\{0,1\}$. Note that $0_a=\bot$ and $0_b=-c$.
\end{example}

Finally, we list a number of properties showing how the distributive lattice $\{ 0_x \mid x \in A \}$ sits inside the algebra $\alg{A}$.

\begin{lemma}\label{lem:monoidalmonoton}
Let $\alg{A} \in \CIdInRL$. For all $x, y \in A$, $x \mleq y$ implies $0_x \mleq 0_y$.
\end{lemma}
\begin{proof}
Consider $x,y \in A$ such that $x \mleq y$. It suffices to show that $0_x \leq 0_y$, as $0_x \cdot 0_y = 0_x \wedge 0_y$. Since $x \mleq y$, we have $y \cdot 0_x = y \cdot x \cdot \negmin x = x \cdot \negmin x = 0_x$. Therefore, $0_x \leq \negmin y$ by $0_x \cdot y = 0_x \leq 0$ and residuation, and $0_x = 0_x \cdot y \leq 1 \cdot y = y$. So $0_x \leq y \wedge \negmin y = 0_y$, as required.
\end{proof}

\begin{lemma}\label{lem:fusioncongruence}
Let $\alg{A} \in \CIdInRL$ and $x,y \in A$. Then the following properties hold:
	\begin{enumerate}
		\item $0_x \cdot 0_y = 0_{x \cdot y}$;
		\item $1_x \cdot 1_y = 1_{x \cdot y}$.
	\end{enumerate}
\end{lemma}
\begin{proof}
(1) Note that $\negmin x \cdot \negmin y \cdot x \cdot y = \negmin x \cdot x \cdot \negmin y \cdot y \leq 0 \cdot 0 = 0$, so $\negmin x \cdot \negmin y \leq \negmin (x \cdot y)$. Hence $x \cdot \negmin x \cdot y\cdot  \negmin y = x \cdot y \cdot \negmin x \cdot \negmin y \leq x \cdot y \cdot \negmin (x \cdot y)$, i.e. $0_x \cdot 0_y \leq 0_{x \cdot y}$.
	
\noindent For the other direction, note that $y \cdot \negmin (x \cdot y) \leq \negmin x$ as $x \cdot y \cdot \negmin (x \cdot y) \leq 0$. Therefore $x \cdot y \cdot \negmin(x \cdot y) \leq x \cdot \negmin x$, i.e. $0_{x \cdot y} \leq 0_x$. An analogous proof shows that $0_{x \cdot y} \leq 0_y$ and so $0_{x \cdot y} \leq 0_x \wedge 0_y = 0_x \cdot 0_y$.
	
(2) This follows from (1) and the De Morgan laws.
\end{proof}

Note that Lemma~\ref{lem:fusioncongruence} proves that the equivalence relation $\equiv_\alg{A}$ that partitions $A$, defined as $x \equiv_\alg{A} y \vcentcolon\Leftrightarrow 0_x = 0_y$, is compatible with $\cdot$. Compatibility with $\negmin$ follows by each $B_x$ being closed under $\negmin$. One might think that $\equiv_\alg{A}$ is a congruence. This is not the case, however. In particular, $\equiv_\alg{A}$ might fail to be compatible with the join $\vee$, as is witnessed by the algebra $\alg{A}_1$ in Figure~\ref{fig:A1}. Note that $\top \equiv_{\alg{A}_1} a$ as well as $-c \equiv_{\alg{A}_1} -c$. But, $-c \vee \top = \top$ and $-c \vee a = b$, so $-c \vee \top \not\equiv_{\alg{A}_1} -c \vee a$.

\section{Congruences and Monoidal Filters}\label{sec:filters}

In this section, we study the sets $\{ x \in A \mid a \mleq x \}$ for a negative element $a \in A^{-}$. They will be instrumental in the gluing construction given in the next section. To motivate the prominent place these sets have, we show that they arise naturally when studying congruences for the variety $\CIdInRL$.

Consider an algebra $\alg{A} \in \CIdInRL$. We call a subset $S \subseteq A$ \emph{convex} if for all $x,y \in S$, $z \in A$, $x \leq z \leq y$ implies $z \in S$. We say that a subset $S \subseteq A$ is a \emph{0-free subuniverse} of $\alg{A}$ if $S$ is closed under all operations $\wedge, \vee, \cdot, \to$ and $1 \in S$, but does not necessarily contain $0$. Note that this means that $S$ is not necessarily closed under $\negmin$. We say that $S$ is a \emph{pointed subuniverse} if it is a 0-free subuniverse with $0 \in S$. 

It is well-known that for commutative pointed residuated lattices, there exists a one-to-one correspondence between the lattice of congruences $\textsf{Con}(\alg{A})$ and the lattice of convex 0-free subuniverses $\mathcal{C}(\alg{A})$.
Furthermore, such convex 0-free subuniverses in turn correspond to convex submonoids of the negative cone. We denote the lattice of all such convex submonoids of $\alg{A^-}$ by $\mathcal{CM}(\alg{A^-})$. We refer to~\cite{GJKO07} for details.

\begin{theorem}[{\cite[Theorem 3.47]{GJKO07}}]\label{thm:congequiv}
Let $\alg{A} \in \CIdInRL$. Then
\[
\textup{\textsf{Con}}(\alg{A}) \cong \mathcal{C}(\alg{A}) \cong \mathcal{CM}(\alg{A}^-),
\]
as witnessed by the following isomorphisms
\begin{align*}
&\textup{\textsf{Con}}(\alg{A}) \to \mathcal{C}(\alg{A});	&& \Theta \mapsto H_\Theta \coloneqq [1]_\Theta \\
&\mathcal{C}(\alg{A}) \to \textup{\textsf{Con}}(\alg{A});	&& H \mapsto \Theta_H \coloneqq \{ (x, y) \in A^2 \mid \text{there exists } h \in H \\
&	&& \qquad \qquad \qquad \text{such that } h \cdot x \leq y \text{ and } h \cdot y \leq x \} \\
&\mathcal{C}(\alg{A}) \to \mathcal{CM}(\alg{A}^-);		&& H \mapsto S_H \coloneqq H^- \\
&\mathcal{CM}(\alg{A}^-) \to \mathcal{C}(\alg{A});		&& S \mapsto H_S \coloneqq \{ x \in A \mid a \leq x \leq a \to 1 \text{ for some } a \in S \}.
\end{align*}
\end{theorem}

As was noted by Stanovsk\'y in \cite{S07}, the convex submonoids of the negative cone of a commutative idempotent residuated lattice are exactly the filters of the negative cone. Formally, a subset $F \subseteq A$ of an $\alg{A} \in \CIdInRL$ is called a \emph{filter} if it is upwards closed under the lattice order $\leq$ as well as closed under $\wedge$. Let $\textsf{Fil}(\alg{A}^-)$ denote the lattice of filters on $\alg{A^-}$. The result by Stanovsk\'y then states that $\mathcal{CM}(\alg{A}^-) = \textsf{Fil}(\alg{A}^-)$ and hence, $\textsf{Con}(\alg{A}) \cong \textsf{Fil}(\alg{A}^-)$.

In this work we are particularly interested in the case when $\alg{A}$ is finite. In such case, each filter $F \in \textsf{Fil}(\alg{A}^-)$ is \emph{principal}, i.e., generated by the single element $\bigwedge F$, and we obtain $\textsf{Con}(\alg{A}) \cong (\alg{A}^-)^\partial$, the dual of $\alg{A}^-$. In light of the isomorphisms from Theorem~\ref{thm:congequiv} above, this means in particular that each set $\{ x \in A \mid a \leq x \leq a \to 1 \}$ for $a \in A^-$ corresponds to a congruence. The rest of this section will be dedicated to the study of these intervals. To start off, the following lemma gives two alternative characterizations.

\begin{lemma}\label{joinclosed}
Let $\alg{A} \in \CIdInRL$ and $a \in A^-$. Then for all $x \in A$,
\[
a \leq x \leq a \to 1 \quad \iff \quad a \leq x \leq 1_a \quad \iff \quad a \mleq x.
\]
\end{lemma}
\begin{proof}
First note that $0_a = a \cdot 0$. One direction follows as $a \leq 1$ implies $0 \leq \negmin a$ and so $a \cdot 0 \leq a \cdot \negmin a = 0_a$. The other direction follows from the fact that $0_a \leq 0$ and $0_a \leq a$ imply $0_a \leq a \wedge 0 = a \cdot 0$. This means that $1_a = \negmin (a \cdot 0) = a \to 1$.

For the other equivalence, we show that $a \mleq x$ if and only if $a \leq x \leq a \to 1$. Suppose that $a \mleq x$. Then $a = a \cdot x \leq x$ as well as $a \cdot x = a \leq 1$, so $x \leq a \to 1$. For the opposite direction, suppose that $a \leq x \leq a \to 1$. From $a \leq x$ we obtain that $a = a \cdot a \leq a \cdot x$. From $x \leq a \to 1$, we obtain $a \cdot x \leq 1$ and so $a \cdot x = a \cdot a \cdot x \leq a$. Hence, $a \cdot x = a$ and so $a \mleq x$.
\end{proof}
\noindent From here on out we will freely switch between these equivalent formulations without mention of this lemma. 

By the isomorphisms in Theorem~\ref{thm:congequiv} we already know that for a filter $F \in \textsf{Fil}(\alg{A}^-)$, the set $\{ x \in A \mid a \mleq x$ for some $a \in F \}$ is a convex 0-free subuniverse of $\alg{A}$. In particular, the set $\{ x \in A \mid a \mleq x \}$ is a convex 0-free subuniverse of $\alg{A}$ for each $a \in A^-$. We characterize exactly when this set forms a pointed subuniverse.

\begin{lemma}\label{lem:pointedsubuniverse}
Let $\alg{A} \in \CIdInRL$ and $a \in A^-$. Then the following conditions are equivalent:
\begin{enumerate}
	\item $a \leq 0$;
	\item for all $ x \in A$, $a \mleq x$ implies $a \mleq \negmin x$;
	\item there exists $x \in A$ such that $a \mleq x$ and $a \mleq \negmin x$.
\end{enumerate}
\end{lemma}
\begin{proof}
$(1) \Rightarrow (2)$ Suppose that $a \leq 0$ and note that $0_a = a$. Then $0_a = a \leq x \leq 1_a$ implies $a = 0_a = \negmin 1_a \leq \negmin x \leq \negmin 0_a = 1_a$.

$(2) \Rightarrow (3)$ Note that $a \mleq a$, so then (2) yields that $a \mleq \negmin a$.

$(3) \Rightarrow (1)$ Suppose there exists $x \in A$ such that $a \mleq x$ and $a \mleq \negmin x$. Then $a \mleq x \cdot \negmin x = 0_x$. As $a, 0_x \in A^-$, this yields $a \leq 0_x \leq 0$.
%
\end{proof}


\begin{corollary}
Let $\alg{A} \in \CIdInRL$ and $a \in A^-$. Then the set $\{ x \in A \mid a \mleq x \}$ is a pointed subuniverse of $\alg{A}$ if and only if $a \leq 0$.
\end{corollary}
\begin{proof}
As noted, $\{ x \in A \mid a \mleq x \}$ is a convex $0$-free subuniverse. The equivalence between (1) and (2) of Lemma~\ref{lem:pointedsubuniverse} yields that it is closed under $\negmin$ if and only if $a \leq 0$.
\end{proof}

In light of the construction outlined in the next section, the case where $a \in A^-$ but $a \not\leq 0$ is of interest. In this case, Lemma~\ref{lem:pointedsubuniverse} above implies that 
\[
\{ x \mid a \mleq x\in A \} \cap \{ \negmin x \mid a \mleq x \in A\} = \emptyset.
\]
We show that the sets $\{ x  \mid a \mleq x\in A \}$ and $\{ \negmin x \mid a \mleq x\in A \}$ are in bijection with one another and show a number of preservation properties of these bijections. Apart from the obvious order-reversing bijection $x \mapsto \negmin x$, we show that the following functions are order-preserving bijections.
\begin{align*}
&(\_) \wedge \negmin a \colon \{ x \mid a \mleq x\in A \} \to \{ \negmin x \mid a \mleq x\in A \} \\
&(\_) \vee a \colon \{ \negmin x \mid a \mleq x\in A \}  \to \{ x \mid a \mleq x\in A \}.
\end{align*}

\begin{lemma}\label{lem:absorption}
Let $\alg{A} \in \CIdInRL$, $a \in A^-$, and $x \in A$. If $x \leq \negmin a$, then $(x \vee a) \wedge \negmin a = x \vee 0_a$. Furthermore, if $a \mleq x$, then $(\negmin x \vee a) \wedge \negmin a = \negmin x$, $(x \wedge \negmin a) \vee a = x$, $a \mleq \negmin (x \wedge \negmin a)$, and $a \mleq \negmin x \vee a$. That is, the maps described above are order-preserving bijections.
\end{lemma}
\begin{proof}
Suppose that $x \leq \negmin a$ for $a \in A^-$. Then it easily follows that $x \vee 0_a \leq \negmin a \vee 0_a = \negmin a$ and $x \vee 0_a \leq x \vee a$ and so $x \vee 0_a \leq (x \vee a) \wedge \negmin a$. For the other direction, consider any $z \in A$ such that $z \leq (x \vee a) \wedge \negmin a$, i.e. $z \leq x \vee a$ and $z \leq \negmin a$. From $z \leq \negmin a$ we deduce that $z \cdot a \leq \negmin a \cdot a = 0_a$ and $x \leq \negmin a$ implies that $x \cdot a \leq \negmin a \cdot a = 0_a$. Then $z \leq x \vee a$ implies that
\[
z \cdot x \leq (x \vee a) \cdot x = x \vee (a \cdot x) \leq x \vee 0_a,
\]
so
\[
z = z \cdot z \leq z \cdot (x \vee a) = (z \cdot x) \vee (z \cdot a) \leq (x \vee 0_a) \vee 0_a = x \vee 0_a.
\]
So indeed, $(x \vee a) \wedge \negmin a = x \vee 0_a$. If $a \mleq x$, then $0_a\leq -x\leq -a$ by Lemma~\ref{joinclosed}. Replacing $x$ by $-x$, it follows that $(\negmin x \vee a) \wedge \negmin a = \negmin x \vee 0_a = \negmin x$. By the De Morgan laws, we then also get $(x \wedge \negmin a) \vee a = x$. Furthermore, we obtain $\negmin x \leq \negmin a \leq 1_a$, so $a \leq \negmin x \vee a \leq 1_a$, i.e., $a \mleq \negmin x \vee a$. By the DeMorgan laws, we also get $a \mleq \negmin (x \wedge \negmin a)$. 
\end{proof}

We list a number of preservation properties here, which will turn out to be useful in the constructions outlined in the following sections.

\begin{lemma}\label{lem:fusionpreserv}
Let $\alg{A} \in \CIdInRL$ and $a \in A^-$.
\begin{enumerate}
\item For all $x, y \in A$ such that $x, y \leq 1_a$, $(x \cdot y) \vee a = (x \vee a) \cdot (y \vee a)$.
\item For all $x, y \in A$ such that $a \mleq x$ and $a \mleq y$, $(x \cdot y) \wedge \negmin a = (x \wedge \negmin a) \cdot (y \wedge \negmin a)$.
\end{enumerate}
\end{lemma}
\begin{proof}
(1) Note that as $x \leq 1_a = a \to 1$, $x \cdot a \leq 1$ and so $x \cdot a = x \cdot a \cdot a \leq a$. Similarly, $a \cdot y \leq a$. So it follows that
\begin{align*}
(x \vee a) \cdot (y \vee a) &= (x \cdot y) \vee (x \cdot a) \vee (a \cdot y) \vee (a \cdot a) \\
	&= (x \cdot y) \vee (x \cdot a) \vee (a \cdot y) \vee a \\
	&= (x \cdot y) \vee a.
\end{align*}

(2) Firstly note that $(x \wedge \negmin a) \cdot a \leq \negmin a \cdot a = 0_a \leq a$ and similarly $(y \wedge \negmin a) \cdot a \leq a$. Hence, by Lemma~\ref{lem:absorption},
\begin{align*}
x \cdot y &= [(x \wedge \negmin a) \vee a] \cdot [(y \wedge \negmin a) \vee a] \\
	&= [(x \wedge \negmin a) \cdot (y \wedge \negmin a)] \vee [(x \wedge \negmin a) \cdot a] \vee [(y \wedge \negmin a) \cdot a] \vee [a \cdot a] \\
	&= [(x \wedge \negmin a) \cdot (y \wedge \negmin a)] \vee a.
\end{align*}
As also $0_a \leq (x \wedge \negmin a) \cdot (y \wedge \negmin a) \leq \negmin a$, Lemma~\ref{lem:absorption} gives
\begin{align*}
(x \cdot y) \wedge \negmin a &= [(x \wedge \negmin a) \cdot (y \wedge \negmin a) \vee a] \wedge \negmin a \\
	&= (x \wedge \negmin a) \cdot (y \wedge \negmin a). \qedhere
\end{align*}
\end{proof}

\section{Gluing Construction}\label{sec:gluing}

In this section we outline a construction to obtain a new member $\alg{A} \oplus_\varphi \alg{B}$ of $\CIdInRL$ from two algebras $\alg{A}, \alg{B} \in \CIdInRL$. In the next section we show how to reverse this construction for each finite $\alg{C} \in \CIdInRL$, allowing for a structural characterization of all finite members of $\CIdInRL$, the main result of this paper. 

Intuitively, we can think of the construction as follows: two algebras $\alg{A}$ and $\alg{B}$ are eligible for the construction if an upset of $\langle A, \mleq^A \rangle$ and a downset of $\langle B, \mleq^B \rangle$ (satisfying some properties) are isomorphic, implemented by a map $\varphi$. The monoidal semilattice of $\alg{A} \oplus_\varphi \alg{B}$ is then constructed by ``placing $\langle B, \mleq^B \rangle$ on top of $\langle A, \mleq^A \rangle$, connected through $\varphi$''. The lattice order is slightly more complicated, but is best expressed by ``wrapping the lattice of $\alg{B}$ inside the lattice of $\alg{A}$'', again connected through $\varphi$ in some way. The involution of $\alg{A} \oplus_\varphi \alg{B}$ is simply the union of that of $\alg{A}$ and $\alg{B}$. For a visual example, we refer to Figure~\ref{fig:gluing}. 

Formally, the ingredients are as follows:
\begin{itemize}
\item two members $\alg{A} = \langle A, \wedge^A, \vee^A, \cdot^A, \to^A, 1^A, 0^A \rangle$ and $\alg{B} = \langle B, \wedge^B, \vee^B, \cdot^B,$ $\to^B, 1^B, 0^B \rangle$ of $\CIdInRL$ such that $A \cap B = \emptyset$;
\item an element $a \in A^-$ such that $a\nleq 0^A$;
\item an element $b \in B^-$;
\item a function $\varphi \colon \{ x \in A \mid a \mleq^A x \} \to \{ y \in B \mid y \mleq^B b \}$ such that
	\begin{itemize}
		\item[$\circ$] $\varphi$ is a bijection;
		\item[$\circ$] $\varphi$ preserves the fusion, i.e. $\varphi(x \cdot^A y) = \varphi(x) \cdot^B \varphi(y)$;
		\item[$\circ$] $\varphi$ preserves the join operation, i.e. $\varphi(x \vee^A y) = \varphi(x) \vee^B \varphi(y)$;
		\item[$\circ$] $0_b = \varphi(a \vee^A 0^A)$.
	\end{itemize}
\end{itemize}
Let us note a few straightforward properties of such ingredients. Firstly, $\alg{B}$ must necessarily be bounded (with respect to the lattice order) as we require $\varphi$ to be a bijection. Secondly, $\varphi$ is an order-isomorphism with respect to both the lattice and the monoidal order, since it preserves both $\vee$ and $\cdot$. Thirdly, $\varphi^{-1}$ also preserves $\vee$ and $\cdot$.  Finally, note that the domain of $\varphi$ is closed under $\vee^A$ by Lemma~\ref{joinclosed}. Similarly, $a\vee^A 0^A$ is in the domain of $\varphi$ since $0_a\le a\le 1^A$ implies $0^A=-1^A\le -0_a=1_a$, hence $a\le a\vee^A 0^A\le 1_a$, or equivalently by Lemma~\ref{joinclosed}, $a\mleq a\vee^A 0^A$.

With the ingredients as listed, we define a new algebra $\alg{A} \oplus_\varphi \alg{B} \coloneqq \langle A \cup B, \wedge, \vee, \cdot, \to, \negmin, 1, 0 \rangle$, referred to as the \emph{gluing} of $\alg{A}$ and $\alg{B}$, where the necessary operations are defined as follows for $x, y \in A \cup B$
\[
x \cdot y = y \cdot x = \begin{cases}
	x \cdot^A y		& x, y \in A \\
	x \cdot^B y		& x, y \in B \\
	x \cdot^A \varphi^{-1}(y \cdot^B b)		& x \in A,~ y \in B ,
\end{cases}
\]
\[
x \vee y = y \vee x = \begin{cases}
	x \vee^A y		& x, y \in A \\
	x \vee^B y		& x, y \in B \\
	\varphi(x \vee^A a) \vee^B y	& x \in A,~ y \in B,~ x \leq^A \negmin^A a \\
	x \vee^A \varphi^{-1}(y \cdot^B b)		& x \in A,~ y \in B,~ x \not\leq^A \negmin^A a,
\end{cases}
\]
\[
\negmin x = \begin{cases}
	\negmin^A x	& x \in A \\
	\negmin^B x	& x \in B,
\end{cases}
\qquad 0 = 0^B, \qquad 1 = 1^B
\]
and the termdefinable connectives $x \to y = \negmin (\negmin y\cdot x)$ as well as $x \wedge y = \negmin (\negmin x \vee \negmin y)$. Note that in the case when $x\le^A -^Aa$, the bijection $\varphi$ is defined for $x\vee^A a$, since $a\le x\vee^A a\le -^Aa\vee a=1_a$. Also, $y\cdot^B b$ is in the domain of $\varphi^{-1}$ because $y\cdot^B b\mleq b$.

The goal of this section is to show that $\alg{A} \oplus_\varphi \alg{B}$ is a member of $\CIdInRL$. To do so, we show that $\langle A \cup B, \vee, \cdot, 1, \negmin\rangle$ is a commutative idempotent \textsf{InRL}-semiring. First we show that $\vee$ is indeed a join operation and deduce the order from this. Since $\alg{A}$ and $\alg{B}$ are disjoint and $\cdot, \vee, -$ are extensions
of these operations on $\alg{A},\alg{B}$ we drop the superscripts when they can be inferred from context.

\begin{lemma}
The operations $\cdot$ and $\vee$ are associative, commutative and idempotent.
\end{lemma}
\begin{proof}
It is easy to see that $\cdot, \vee$ are idempotent and commutative, since this holds for $\cdot^A,\cdot^B, \vee^A$ and $\vee^B$. To show that $\cdot$ is associative, consider $x,y,z \in A \cup B$. We distinguish cases based on $x,y,z$ being members of $A$ or $B$. If $x,y,z\in A$ or $x,y,z\in B$
then this follows from $\cdot^A$ and $\cdot^B$ being associative. Commutativity implies that
$(x \cdot y) \cdot z = x \cdot (y \cdot z)$ implies $z \cdot (y \cdot x) = (z \cdot y) \cdot x$, hence we only have to check four cases.

Case $x,y\in A,~ z\in B$: $(x{\cdot} y){\cdot} z=(x{\cdot} y){\cdot}\varphi^{-1}(z{\cdot} b)=x{\cdot} (y{\cdot}\varphi^{-1}(z{\cdot} b))$ $=x{\cdot} (y{\cdot} z)$.

Case $x\in A,~y,z\in B$: $(x\cdot y)\cdot z=(x\cdot \varphi^{-1}(y\cdot b))\cdot\varphi^{-1}(z\cdot b)$ $=x\cdot \varphi^{-1}(y\cdot z\cdot b)=x\cdot (y\cdot z)$ where we made use of the fact that $\varphi^{-1}$ preserves $\cdot$ since it is the inverse of a map preserving $\cdot$.

The remaining cases $x,z\in A$, $y\in B$ and $y\in A$, $x,z\in B$ are similar.

For the associativity of $\vee$ we proceed similarly, but the four cases are doubled or quadrupled
depending on whether the members of $A$ are less-or-equal to $-a$ or not. We cover two of the cases that are less straightforward.

Suppose that $x,z \in A$, $y \in B$, $x \leq \negmin a$ and $z \not\leq \negmin a$. Note that $\varphi^{-1}$ preserves the join operation since $\varphi$ is a join homomorphism, and for any $w\in B$
we have $w\cdot b\mleq b$, hence $a \leq \varphi^{-1}(w \cdot b)$. Thus
\begin{align*}
(x \vee y) \vee z		&= (\varphi(x \vee a) \vee y) \vee z \\
	&= \varphi^{-1}((\varphi(x \vee a) \vee y) \cdot b) \vee z \\
	&= \varphi^{-1}(\varphi(x \vee a) \cdot b \vee y \cdot b) \vee z \\
	&= \varphi^{-1}(\varphi(x \vee a) \vee y \cdot b) \vee z \\
	&= \varphi^{-1}\varphi(x \vee a) \vee \varphi^{-1}(y \cdot b) \vee z \\
	&= x \vee a \vee \varphi^{-1}(y \cdot b) \vee z \\
	&= x \vee \varphi^{-1}(y \cdot b) \vee z \\
	&= x \vee (y \vee z).
\end{align*}

Suppose that $x \in A$, $y,z \in B$, $x \not\leq \negmin a$. Then again using the fact that $\varphi^{-1}$ preserves the join,
\begin{align*}
(x \vee y) \vee z		&= (x \vee \varphi^{-1}(y \cdot b)) \vee z \\
	&= x \vee \varphi^{-1}(y \cdot b) \vee \varphi^{-1}(z \cdot b) \\
	&= x \vee \varphi^{-1}(y \cdot b \vee z \cdot b) \\
	&= x \vee \varphi^{-1}((y \vee z) \cdot b) \\
	&= x \vee (y \vee z). \qedhere
\end{align*}
\end{proof}

Now that we have shown that $\vee$ is a join operation, it is easily verified that the corresponding lattice order $\leq$ can be expressed as follows:
\[
x \leq y \iff \begin{cases}
	x \leq^A y		& x, y \in A \\
	x \leq^B y		& x, y \in B \\
	\varphi(x \vee^A a) \leq^B y \text{ and } x \leq^A \negmin^A a		& x \in A, y \in B \\
	\varphi^{-1}(x \cdot^B b) \leq^A y\text{ and } y\nleq^A-^A a		& x \in B, y \in A.
\end{cases}
\]

\noindent The fact that $\cdot$ distributes over $\vee$ follows from residuation. Hence, we are left to show the residuation law. For this proof, we need two small facts.

\begin{lemma}\label{lem:factsconstruction}
\begin{enumerate}
\item For all $x \in B$, $y \in A$, $\varphi^{-1}(x \cdot b) \cdot y \leq \negmin a$ iff $y \leq \negmin a$.
\item For all $z \in A$, $z \leq 0$ iff $z \leq 0^A$.
\end{enumerate}
\end{lemma}
\begin{proof}
(1) As $a \mleq \varphi^{-1}(x \cdot b)$, we have that $\varphi^{-1}(x \cdot b) \cdot y \leq \negmin a$ is equivalent to $a \cdot y = a \cdot \varphi^{-1}(x \cdot b) \cdot y \leq 0^A$, i.e. $y \leq \negmin a$.

(2) For the left-to-right direction, suppose that $z \leq 0$, that is, $\varphi(z \vee a) \leq 0^B$ and $z \leq \negmin a$. As $\varphi(z \vee a) \mleq b$ and $b \leq 1^B$,
\[
\varphi(z \vee a) = \varphi(z \vee a) \cdot b \leq 0^B \cdot b \leq \negmin b \cdot b = 0_b.
\]
Therefore, because $0_b = \varphi(0^A \vee a)$ and $\varphi$ reflects the lattice order, $z \vee a \leq 0^A \vee a$. Hence, $z \leq (z \vee a) \wedge \negmin a \leq (0^A \vee a) \wedge \negmin a = 0^A$ by Lemma~\ref{lem:absorption}.

For the right-to-left direction, suppose that $z \leq 0^A$. Then $z \leq 0^A \leq \negmin a$ (since $a\leq 1^A$ is a standing assumption) as well as $z \vee a \leq 0^A \vee a$. We note that for $z \leq 0^A$, we have $a\mleq z\vee a$
because $a\leq z \vee a\leq -a\vee a=1_a$, and similarly $a\mleq 0^A\vee a$ hence $\varphi$ is defined at $z\vee a$ and $ 0^A\vee a$.  As $\varphi$ preserves the lattice order, $\varphi(z \vee a) \leq \varphi(0^A \vee a) = 0_b \leq 0^B$ and so $z \leq 0$.
\end{proof}

Finally, we show the required residuation property.

\begin{lemma}
For all $x, y \in A \cup B$,
\[
x \cdot y \leq 0 	\quad \iff \quad x \leq \negmin y.
\]
\end{lemma}
\begin{proof}
Consider $x, y \in A \cup B$. Again, we prove by cases. The case when $x,y \in A$ follows from Lemma~\ref{lem:factsconstruction}(2) together with residuation in $\alg{A}$. The case for $x,y \in B$ follows directly from residuation in $\alg{B}$. 

Next, suppose that $x \in A$ and $y \in B$. Then $x \cdot y = x \cdot \varphi^{-1}(y \cdot b)$, so
\begin{align*}
x \cdot y \leq 0 &\iff \varphi((x \cdot \varphi^{-1}(y \cdot b)) \vee a) \leq 0^B \\
&\qquad \qquad \text{and } x \cdot \varphi^{-1}(y \cdot b) \leq \negmin a
\end{align*}
and
\begin{align*}
x \leq \negmin y	&\iff \varphi(x \vee a) \leq \negmin y \text{ and } x \leq \negmin a \\
	&\iff y \cdot \varphi(x \vee a) \leq 0^B \text{ and } x \leq \negmin a.
\end{align*}
Note that by Lemma~\ref{lem:factsconstruction}(1), we have that 
\[
x \cdot \varphi^{-1}(y \cdot b) \leq \negmin a \ \iff \ x \leq \negmin a.
\]
So assume that $x \leq \negmin a$. Since also $a \mleq \varphi^{-1}(y \cdot b)$, Lemma~\ref{lem:fusionpreserv}(1) implies that
\begin{align*}
(x \cdot \varphi^{-1}(y \cdot b)) \vee a 	&= (x \vee a) \cdot (\varphi^{-1}(y \cdot b) \vee a) \\
&= (x \vee a) \cdot \varphi^{-1}(y \cdot b).
\end{align*}
Applying $\varphi$ to both sides then gives
\begin{align*}
\varphi((x \cdot \varphi^{-1}(y \cdot b)) \vee a)	&= \varphi((x \vee a) \cdot \varphi^{-1}(y \cdot b)) \\
&= \varphi(x \vee a) \cdot (y \cdot b) \\
&= \varphi(x \vee a) \cdot y,
\end{align*}
proving the required equivalence.

Finally, suppose $y\in A$ and $x\in B$. Then $x\cdot y\in A$, hence $x\cdot y\le 0$ is equivalent to
$x\cdot y\le 0^A$ by Lemma~\ref{lem:factsconstruction}(2). By definition of $\cdot$, we also have
$x\cdot y=y\cdot\varphi^{-1}(x\cdot b)$.

By definition of $\leq$ we have $x\leq \negmin y$ if and only if $\varphi^{-1}(x\cdot b)\leq \negmin y$ and $\negmin y\nleq\negmin a$. The first condition is equivalent to $y\cdot\varphi^{-1}(x\cdot b)\leq 0^A$, and we claim that it subsumes $\negmin y\nleq\negmin a$. Suppose to the contrary that $\negmin y\leq\negmin a$, i.e., $a\leq y$. Since we also have $a\le \varphi^{-1}(x\cdot b)\leq 1_a$, it follows that $a=a \cdot a\leq y\cdot \varphi^{-1}(x\cdot b)\leq 0^A$. This is a contradiction since $a\nleq 0^A$ is one of the ingredients of the $\alg A\oplus_\varphi\alg B$ construction. 
\end{proof}

We have shown that $\alg{A} \oplus_\varphi\alg{B}$ satisfies all required properties of a commutative idempotent \textsf{InRL}-semiring. By the term equivalence from Theorem~\ref{thm:termequiv}, we obtain the following theorem.

\begin{theorem}\label{thm:gluing}
The algebra $\alg{A} \oplus_\varphi \alg{B}$ is a member of $\CIdInRL$.
\end{theorem}

\begin{example}\label{exa:gluing}
A non-trivial example of the gluing construction as outlined here is given in Figure~\ref{fig:gluing}, where the algebra $\alg{A} \oplus_\varphi \alg{B}$ is obtained by gluing the algebra 
$\alg{A}$ with universe $A = \{ x \in A \cup B \mid x \mleq 1_u \}$ and the algebra $\alg{B}$ with universe $B = \{ y \in A \cup B \mid 0_v \mleq y \}$. The bijection $\varphi$, depicted by dashed lines in $\langle \alg{A} \oplus_\varphi \alg{B}, \mleq \rangle$, is defined by $\varphi(1_u)=b$, $\varphi(u)=0_b$, $\varphi(1_a)=v$, and $\varphi(a)=0_v$.

\begin{figure}[h]
\begin{center}
\begin{tikzpicture}[baseline=0pt,scale=0.59]
	
	\node at (-0.8,-1)[n] {$\langle \alg{A} \oplus_\varphi \alg{B}, \mleq \rangle$};
	
	\node(0w) at (0,0)[label=left:$0_w$]{};
	\node(-w)	at (-1,1)[label=left:$\negmin w$]{};
	\node(w)	at (1,1)[label=right:$w$]{};
	\node(1w)	at (0,2)[label=right:$1_w$]{};
	
	\node(0a)	at (-2,2)[label=left:$0_a$]{};
	\node(-a)	at (-3,3)[label=left:$\negmin a$]{};
	\node(a)	at (-1,3)[label=right:$a$]{};
	\node(1a)	at (-2,4)[label=right:$1_a$]{};
	
	\node(0u)	at (-4,4)[label=left:$0_u$]{};
	\node(-u)	at (-5,5)[label=left:$\negmin u$]{};
	\node(u)	at (-3,5)[label=right:$u$]{};
	\node(1u)	at (-4,6)[label=right:$1_u$]{};
	
	\node(0v)	at (0,4)[label=right:$0_v$]{};
	\node(v)	at (-1,5)[label=left:$v$]{};
	\node(-v)	at (1,5)[label=right:$\negmin v$]{};
	\node(1v)	at (0,6)[label=right:$1_v$]{};
	
	\node(0b)	at (-2,6)[label=left:$0_b$]{};
	\node(b)	at (-3,7)[label=left:$b$]{};
	\node(-f)	at (-1,7)[]{};
	\node(g)	at (-2,8)[]{};
	\node(-g)	at (-2,7)[]{};
	\node(f)	at (-3,8)[]{};
	\node(-b)	at (-1,8)[label=right:$\negmin b$]{};
	\node(1b)	at (-2,9)[label=left:$1_b$]{};
	
	\node[n]	at (-2,5) {$\varphi$};
	
	\draw (0w) edge[boolean] (w) edge[boolean] (-w);
	\draw (w) edge[boolean] (1w);
	\draw (-w) edge (0a) edge[boolean] (1w);
	\draw (1w) edge (a);
	\draw (0a) edge[boolean] (-a) edge[boolean] (a);
	\draw (a) edge[boolean] (1a);
	\draw (-a) edge (0u) edge[boolean] (1a);
	\draw (1a) edge (u);
	\draw (0u) edge[boolean] (-u) edge[boolean] (u);
	\draw (u) edge[boolean] (1u);
	\draw (-u) edge[boolean] (1u);
	
	\draw (0v) edge[boolean] (v) edge[boolean] (-v);
	\draw (-v) edge[boolean] (1v);
	\draw (v) edge (0b) edge[boolean] (1v);
	\draw (1v) edge (-f);
	\draw (0b) edge[boolean] (-f) edge[boolean] (-g) edge[boolean] (b);
	\draw (-f) edge[boolean] (g) edge[boolean] (-b);
	\draw (-g) edge[boolean] (f) edge[boolean] (-b);
	\draw (-b) edge[boolean] (1b);
	\draw (b) edge[boolean] (f) edge[boolean] (g);
	\draw (g) edge[boolean] (1b);
	\draw (f) edge[boolean] (1b);
	
	\draw (a) edge[dashed,->] (0v);
	\draw (1a) edge[dashed,->] (v);
	\draw (u) edge[dashed,->] (0b);
	\draw (1u) edge[dashed,->] (b);
\end{tikzpicture}
\quad
\begin{tikzpicture}[baseline=0pt,scale=0.59]

	\node at (-0.8,-1)[n] {$\langle \alg{A} \oplus_\varphi \alg{B}, \leq \rangle$};

	\node(0w)	at (0,0)[label=left:$0_w$]{};
	\node(0a)	at (-1,1)[label=left:$0_a$]{};
	\node(0u)	at (-2,2)[label=left:$0_u$]{};
	\node(-u)	at (-3,3)[label=left:$\negmin u$]{};
	\node(-a)	at (-4.5,4.5)[label=left:$\negmin a$]{};
	\node(-w)	at (-5.5,5.5)[label=left:$\negmin w$]{};
	
	\node(w)	at (3.5,3.5)[label=right:$w$]{};
	\node(a)	at (2.5,4.5)[label=right:$a$]{};
	\node(u)	at (1,6)[label=right:$u$]{};
	\node(1u)	at (0,7)[label=right:$1_u$]{};
	\node(1a)	at (-1,8)[label=right:$1_a$]{};
	\node(1w)	at (-2,9)[label=right:$1_w$]{};
	
	\node(0v)	at (0,2)[label=right:$0_v$]{};
	\node(-v)	at (1,3)[label=right:$\negmin v$]{};
	\node(v)	at (-3,6)[label=left:$v$]{};
	\node(1v)	at (-2,7)[label=left:$1_v$]{};
	
	\node(0b)	at (-1,3)[label=left:$0_b$]{};
	\node(b)	at (-2,4)[label=left:$b$]{};
	\node(-f)	at (0,4)[]{};
	\node(g)	at (-1,5)[]{};
	\node(-g)	at (-1,4)[]{};
	\node(f)	at (-2,5)[]{};
	\node(-b)	at (0,5)[label=right:$\negmin b$]{};
	\node(1b)	at (-1,6)[label=left:$1_b$]{};

	\draw (0w) edge (0a) edge (w);
	\draw (0a) edge (0u);
	\draw (0u) edge (-u);
	\draw (-u) edge (-a);
	\draw (-a) edge (-w);
	\draw (-w) edge (1w);
	
	\draw (w) edge (a);
	\draw (a) edge (u);
	\draw (u) edge (1u);
	\draw (1u) edge (1a);
	\draw (1a) edge (1w);
	
	\draw (0v) edge (0b) edge (-v);
	\draw (-v) edge (-f);
	
	\draw (0b) edge (-f) edge (-g) edge (b);
	\draw (-f) edge (g) edge (-b);
	\draw (-g) edge (f) edge (-b);
	\draw (-b) edge (1b);
	\draw (b) edge (f) edge (g);
	\draw (g) edge (1b);
	\draw (f) edge (1b);
	
	\draw (f) edge (v);
	\draw (1b) edge (1v);
	\draw (v) edge (1v);

	\draw (0a) edge[dashed] (0v);
	\draw (0u) edge[dashed] (0b);
	\draw (-u) edge[dashed] (b);
	\draw (-a) edge[dashed] (v);
	
	\draw (-v) edge[dashed] (a);
	\draw (-b) edge[dashed] (u);
	\draw (1b) edge[dashed] (1u);
	\draw (1v) edge[dashed] (1a);
\end{tikzpicture}
\end{center}
\caption{A depiction of the algebra $\alg{A} \oplus_\varphi \alg{B}$ from Example~\ref{exa:gluing}.}\label{fig:gluing}
\end{figure}
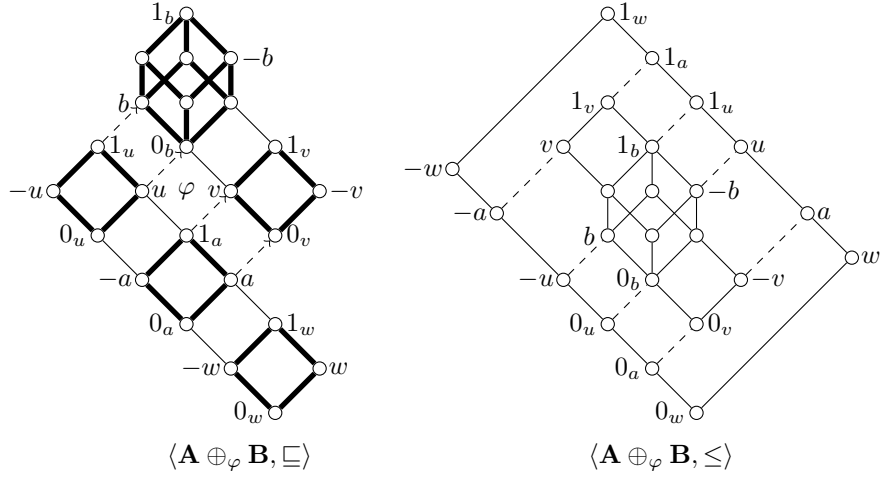
\end{example}

\section{Decomposition}
\label{sec:reverse}

In this section we outline how to reverse the gluing construction given in the previous section for any finite member of $\CIdInRL$. The main result, Theorem~\ref{thm:structchar}, is that we can construct any finite member of the variety $\CIdInRL$ starting from all finite Boolean algebras. 

Consider any $\alg{C} \in \CIdInRL$. To reverse the construction from Section~\ref{sec:gluing}, we find two algebras $\alg{A}, \alg{B} \in \CIdInRL$ and a bijection $\varphi$ such that $\alg{C} = \alg{A} \oplus_\varphi \alg{B}$. A crucial role is reserved for the atoms of the distributive lattice $\langle C^+, \wedge, \vee, 1 \rangle$ from Theorem~\ref{distriblattice}. Recall that an \emph{atom} is an element $a \in C^+$ such that for any $b \in C^+$ such that $1 \leq b \leq a$, either $b = 1$ or $b = a$. For an arbitrary member of $\CIdInRL$, such atoms need not exist. In particular, if $\alg{C}$ is a Boolean algebra, $C^+ = \{ 1 \}$ and therefore $C^+$ contains no atoms. Moreover, if $\alg{C}$ is infinite, atoms of $C^+$ also need not exist. For example, consider an infinite version of the algebras $\alg{A}_n$ constructed in the next section, depicted in Figure~\ref{fig:locfinite}. Therefore we consider $\alg{C}$ to be finite and not a Boolean algebra and let $c \in C^+$ be such an atom. As $\langle C^+, \wedge, \vee, 1 \rangle$ is distributive and $c$ is an atom, there exists a (unique) element $c^\ast \in C^+$ such that 
\[
\{ x \in C^+ \mid x \mleq c \} \cup \{ x \in C^+ \mid c^\ast \mleq x \} = C^+
\]
and
\[
\{ x \in C^+ \mid x \mleq c \} \cap \{ x \in C^+ \mid c^\ast \mleq x \} = \emptyset
\]
(see~\cite[Theorem 6]{Sto38}). These two elements $c$ and $c^\ast$ can then be used to partition $C$ into two subsets: $\{ x \in C \mid x \mleq c \}$ and $\{ y \in C \mid \negmin c^\ast \mleq y \}$. We summarize some properties of these two intervals. Note that $c=1_c$ and $-c^*=0_{c^*}$.

\begin{lemma}
	Let $c$ and $c^\ast$ be as defined above.
	\begin{enumerate}
		\item For all $x \in C$, $\negmin c^\ast \mleq x$ if and only if $x \not\mleq c$;
		\item $\langle \{ x \in C \mid x \mleq c \}, \wedge, \vee, \cdot, \to, c, \negmin c \rangle$ and $\langle \{ y \in C \mid \negmin c^\ast \mleq y \}, \wedge, \vee, \cdot,$ $\to, 1, 0 \rangle$ are a subalgebra and pointed subalgebra of $\alg{C}$, respectively.
	\end{enumerate}
\end{lemma}
\begin{proof}
(1) Consider $x \in C$. For the left-to-right direction, suppose that $\negmin c^\ast \mleq x$. For a contradiction, suppose that $x \mleq c$. But then $\negmin c^\ast \mleq c$, contradicting that $\{ x \in C^+ \mid x \mleq c \} \cap \{ x \in C^+ \mid c^\ast \mleq x \} = \emptyset$.
For the right-to-left direction suppose that $x \not\mleq c$. Since $x\mleq 1_x$, we have $1_x \not\mleq c$ and, by choice of $c^\ast$, $c^\ast \mleq 1_x$. It follows from Lemma~\ref{lem:monoidalmonoton} that $\negmin c^\ast \mleq 0_x$ and hence $\negmin c^\ast \mleq 0_x \mleq x$.

(2) This is an instance of a more general fact: for any $\alg{A} \in \CIdInRL$ and $x,y \in A$ such that $1_x \mleq 1_y$ (or, equivalently, $0_x \mleq 1_y$), $\{ z \in A \mid 0_x \mleq z \mleq 1_y \}$ forms a subalgebra, with constants $0_y$ and $1_y$. Indeed, closure under ${\cdot}$ and ${\lor}$ is straightforward, and closure under ${\negmin}$ follows from the fact that $\{ z \in A \mid 0_x \mleq z \mleq 1_y \} = \bigcup \{ B_z \mid 1_x \mleq z \mleq 1_y \}$. The residuation law can be easily checked.
\end{proof}

We now define algebras $\alg{A} \coloneqq \langle \{ x \in C \mid x \mleq c \}, \wedge, \vee, \cdot, \to, c, \negmin c \rangle$ and $\alg{B} \coloneqq \langle \{ y \in C \mid \negmin c^\ast \mleq y \}, \wedge, \vee, \cdot, \to, 1, 0 \rangle$, elements $a \coloneqq c \cdot \negmin c^\ast$ and $b \coloneqq (c \wedge \negmin a) \vee \negmin c^\ast$ and maps
\begin{itemize}
	\item $\varphi \colon \{ x \in C \mid a \mleq x \mleq c \} \to \{ y \in C \mid \negmin c^\ast \mleq y \mleq b \}$ where $\varphi(x) \coloneqq (x \wedge \negmin a) \vee \negmin c^\ast$;
	\item $\varphi^{-1} \colon \{ y \in C \mid \negmin c^\ast \mleq y \mleq b \} \to \{ x \in C \mid a \mleq x \mleq c \}$ where $\varphi^{-1}(y) \coloneqq y \cdot c$.
\end{itemize}

The rest of the section is dedicated to proving that $\alg{C} = \alg{A} \oplus_\varphi \alg{B}$. We start by showing that the defined elements and map satisfy the prerequisites of the construction.

\begin{lemma}\label{lem:propab}
For the elements $a$ and $b$ as defined above, 
	\begin{enumerate}
		\item $a \in A^-$ and $a\nleq 0^A$;
		\item $b \in B^-$, $c\cdot 0=-c\vee a$ and $0\leq b$.
	\end{enumerate}
\end{lemma}
\begin{proof}
(1) We have $a\leq 1^A=c$ since $a = c \cdot \negmin c^\ast \leq c \cdot 0 \leq c \cdot 1 = c$. For the other claim, assume for a contradiction that $a \leq 0^A = \negmin c$. By residuation $a\cdot c\leq 0$, hence $c \cdot \negmin c^\ast = c \cdot c \cdot \negmin c^\ast = c \cdot a \leq 0$, which implies $c \leq c^\ast$, a contradiction.
	
(2) Note that $b \leq 1$ is equivalent to $c \wedge \negmin a \leq 1$ together with $\negmin c^\ast \leq 1$. The latter inequality follows as $\negmin c^\ast \leq 0 \leq 1$. For the former, note that $-c$ is a coatom in $\{0_x\mid x\in C\}$ by Theorem~\ref{distriblattice}, hence $0=0_c\vee 0_{c^*}=-c\vee -c^*$ and it follows that
$$0 = 1 \cdot 0 \leq c \cdot 0 = c \cdot (\negmin c \vee \negmin c^\ast) = (c \cdot \negmin c) \vee (c \cdot \negmin c^\ast) = \negmin c \vee (c \cdot \negmin c^\ast) = \negmin c \vee a.$$
Therefore $c\wedge -a\leq 1$. To show $0\leq b$, we note that 
$(c\cdot 0)\vee c=c\cdot (0\vee 1)=c$, hence $0=-c\vee -c^*\leq-(c\cdot 0)\vee -c^*=(c\wedge -a)\vee-c^*=b$.
\end{proof}

We need a couple of technical properties, which we prove in a separate lemma.

\begin{lemma}\label{lem:technicalpropAB} 
\begin{enumerate}
	\item For all $y \in B$, $\varphi(y \cdot c) = y \cdot b$.
	\item For all $x \in A$, $y \in B$, $x \leq \negmin a$ if and only if $x \vee y \in B$.
\end{enumerate}
\end{lemma}
\begin{proof}
(1) Consider $y \in B$. Then,
\begin{align*}
\varphi(y \cdot c)	&= ((y \cdot c) \wedge \negmin a) \vee \negmin c^\ast \\
&= \big((y \cdot [(c \wedge \negmin a) \vee a]) \wedge \negmin a \big) \vee \negmin c^\ast \\
&= \big([(y \cdot (c \wedge \negmin a)) \vee (y \cdot a)] \wedge \negmin a \big) \vee \negmin c^\ast \\
&= \big([(y \cdot (c \wedge \negmin a)) \vee a] \wedge \negmin a \big) \vee \negmin c^\ast \\
&= (y \cdot (c \wedge \negmin a)) \vee \negmin c^\ast \\
&= (y \cdot (c \wedge \negmin a)) \vee y \cdot \negmin c^\ast \\
&= y \cdot [(c \wedge \negmin a) \vee \negmin c^\ast] \\
&= y \cdot b.
\end{align*}
The fourth equality follows as $y \cdot a = y \cdot c \cdot \negmin c^\ast = c \cdot \negmin c^\ast = a$. The fifth equality follows since $y \cdot (c \land \negmin a) = ((y \cdot (c \land \negmin a)) \vee a ) \land \negmin a$ by an application of Lemma~\ref{lem:absorption}.

(2) As $\negmin c^\ast \cdot (x \vee y) = (\negmin c^\ast \cdot x) \vee (\negmin c^\ast \cdot y) = (\negmin c^\ast \cdot x) \vee \negmin c^\ast$,  it always holds that $\negmin c^\ast \leq \negmin c^\ast \cdot (x \vee y)$. So $\negmin c^\ast \mleq x \vee y$ is in turn equivalent to $\negmin c^\ast \cdot (x \vee y) \leq \negmin c^\ast$, which is equivalent to $\negmin c^\ast \cdot x \leq \negmin c^\ast$. Via residuation, this is equivalent to $a \cdot x = \negmin c^\ast \cdot c \cdot x = \negmin c^\ast \cdot x = c^\ast \cdot \negmin c^\ast \cdot x \leq 0$ which by residuation again is equivalent to $x \leq \negmin a$.
\end{proof}

\begin{lemma}
The functions $\varphi$ and $\varphi^{-1}$ as defined above are well-defined.
\end{lemma}
\begin{proof}
To show that $\varphi^{-1}$ is well-defined, we assume $y \in C$ satisfies $\negmin c^\ast \mleq y \mleq b$, and we need to show that $a \mleq \varphi^{-1}(y) \mleq c$. It is immediate that $\varphi^{-1}(y)=y \cdot c \mleq c$. Moreover, $a \cdot y \cdot c = c \cdot \negmin c^\ast \cdot y \cdot c = c \cdot \negmin c^\ast=a$ hence $a\mleq\varphi^{-1}(y)$.

To prove that $\varphi(x)=(x\wedge -a)\vee -c^*$ is well-defined, we assume $x \in C$ satisfies $a \mleq x \mleq c$ and show that $\negmin c^\ast \mleq \varphi(x) \mleq b$. Firstly note that since $\negmin c^\ast \leq 1$, $\negmin c^\ast \mleq \varphi(x)$ is equivalent to $\negmin c^\ast \leq \varphi(x) \leq 1_{-c^*}=c^\ast$  by Lemma~\ref{joinclosed}. It is immediate that $\negmin c^\ast \leq (x \wedge \negmin a) \vee \negmin c^\ast = \varphi(x)$. Moreover, note that $\varphi(x) \leq c^\ast$ is equivalent to $\negmin c^\ast \leq c^\ast$ together with $x \wedge \negmin a \leq c^\ast$. Obviously $\negmin c^\ast \leq c^\ast$ holds. For the other statement, note that $1 \leq c$ implies that $\negmin c^\ast \leq c \cdot \negmin c^\ast = a$. It follows that $\negmin a \leq c^\ast$ and so $x \wedge \negmin a \leq \negmin a \leq c^\ast$ as required.

To show that $\varphi(x) \mleq b$, we consider $\varphi(x) \cdot b$.
\begin{align*}
\varphi(x) \cdot b &= [(x \wedge \negmin a) \vee \negmin c^\ast] {\cdot} [(c \wedge \negmin a) \vee \negmin c^\ast] \\
	&= [(x \wedge \negmin a) {\cdot} (c \wedge \negmin a)] \vee [(x \wedge \negmin a) {\cdot} \negmin c^\ast] \vee [(c \wedge \negmin a) {\cdot} \negmin c^\ast] \vee [\negmin c^\ast {\cdot} \negmin c^\ast] \\
	&= [(x \wedge \negmin a) {\cdot} (c \wedge \negmin a)] \vee \negmin c^\ast \\
	&= [(x \cdot c) \wedge \negmin a] \vee \negmin c^\ast \\
	&= [x \wedge \negmin a] \vee \negmin c^\ast \\
	&= \varphi(x).
\end{align*}
The third equality follows from the fact that $(x \wedge \negmin a) \cdot \negmin c^\ast \leq \negmin a \cdot \negmin c^\ast \leq c^\ast \cdot \negmin c^\ast = \negmin c^\ast$, and similarly, $(c \wedge \negmin a) \cdot \negmin c^\ast \leq \negmin c^\ast$. The fourth equality follows from Lemma~\ref{lem:fusionpreserv}(2). 
\end{proof}

\begin{lemma}\label{lem:phibijective}
The function $\varphi$ is a bijection with $\varphi^{-1}$ as its inverse, i.e.,
\begin{enumerate}
	\item for $a \mleq x \mleq c$, $\varphi^{-1}(\varphi(x))= x$;
	\item for $\negmin c^\ast \mleq y \mleq b$, $\varphi(\varphi^{-1}(y))= y$.
\end{enumerate}
\end{lemma}
\begin{proof}
(1) Let $a \mleq x \mleq c$. Then,
	\begin{align*}
	\varphi^{-1}(\varphi(x)) &= \varphi^{-1}((x \wedge \negmin a) \vee \negmin c^\ast) \\
		&= [(x \wedge \negmin a) \vee \negmin c^\ast] \cdot c \\
		&= ((x \wedge \negmin a) \cdot c) \vee (\negmin c^\ast \cdot c) \\
		&= ((x \wedge \negmin a) \cdot c) \vee a \\
		&= (x \wedge \negmin a) \vee a \\
		&= x.
	\end{align*}
The fifth equality follows from the observation that, since $A$ is closed under $\wedge$, $x \wedge \negmin a \mleq c$. The last equality holds by Lemmas~\ref{lem:absorption} and \ref{lem:propab}(1).
	
(2) The desired result follows from Lemma~\ref{lem:technicalpropAB}(1) and $y \mleq b$, because $\varphi(\varphi^{-1}(y)) = \varphi(y \cdot c) = y \cdot b = y$.
\end{proof}

\begin{lemma}
The functions $\varphi$ and $\varphi^{-1}$ satisfy all prerequisites of the construction, i.e.,
\begin{enumerate}
	\item $\varphi$ and $\varphi^{-1}$ preserve $\vee$;
	\item $\varphi$ and $\varphi^{-1}$ preserve $\cdot$;
	\item $\varphi(\negmin c \vee a) = 0_b$.
\end{enumerate}
\end{lemma}
\begin{proof}
Firstly note that as we have already shown that $\varphi$ is a bijection with inverse $\varphi^{-1}$, it suffices to show that $\varphi^{-1}$ preserves $\vee$ and $\cdot$. The function $\varphi^{-1}$ preserves $\vee$ as
\[
\varphi^{-1}(y \vee z) = (y \vee z) \cdot c = (y \cdot c) \vee (z \cdot c) = \varphi^{-1}(y) \vee \varphi^{-1}(z).
\] 
The function $\varphi^{-1}$ preserves $\cdot$ as
\[
\varphi^{-1}(y \cdot z) = y \cdot z \cdot c = (y \cdot c) \cdot (z \cdot c) = \varphi^{-1}(y) \cdot \varphi^{-1}(z).
\]
For the last item, note that by Lemma~\ref{lem:propab}(2), $0\leq b \leq 1$
so $0_b = 0$. Therefore, $\varphi^{-1}(0_b) = \varphi^{-1}(0) = c \cdot 0 = \negmin c \vee a$ and hence $\varphi(\negmin c \vee a) = 0_b$.
\end{proof}

We have now shown that $\alg{A} \oplus_\varphi \alg{B}$ is well-defined. It remains to show that indeed $\alg{A} \oplus_\varphi \alg{B}$ and $\alg{C}$ coincide.

\begin{lemma}
$\alg{C} = \alg{A} \oplus_\varphi \alg{B}$.
\end{lemma}
\begin{proof}
It is obvious that the universes of the two algebras coincide. It easily follows that the involution $\negmin$ and the constants coincide as well. 

For the fusion operation, the interesting case is to show that $x \cdot y = x \cdot \varphi^{-1}(y \cdot b)$ for $x \in A$, $y \in B$. Note that by Lemma~\ref{lem:technicalpropAB}(1), $x \cdot \varphi^{-1}(y \cdot b) = x \cdot \varphi^{-1}(\varphi(y \cdot c)) = x \cdot y \cdot c = x \cdot y$ since $x\mleq c$.

For the join operation, the two interesting cases are when $x \in A$ and $y \in B$. Firstly suppose that $x \leq \negmin a$. Then
\begin{align*}
\varphi(x \vee a) \vee y &= ([(x \vee a) \wedge \negmin a] \vee \negmin c^\ast) \vee y \\
	&= (x \vee 0_a) \vee \negmin c^\ast \vee y \\
	&= x \vee \negmin c^\ast \vee y \\
	&= x \vee y.
\end{align*}
The second equality follows by Lemma~\ref{lem:absorption} and the third by Lemma~\ref{joinclosed}, since $a\mleq c^*$, hence $c^*\leq 1_a$ and therefore $0_a\leq -c^*$. The last equality follows from $\negmin c^\ast = \negmin c^\ast \cdot y \leq 1 \cdot y = y$.

For the other case, suppose that $x \not\leq \negmin a$. By Lemma~\ref{lem:technicalpropAB}(2), this is equivalent to $x \vee y \notin B$, hence $x \vee y \in A$. Then, by Lemma~\ref{lem:technicalpropAB}(1),
\begin{align*}
x \vee \varphi^{-1}(y \cdot b) &= x \vee \varphi^{-1}(\varphi(y \cdot c)) \\
	&= x \vee (y \cdot c) \\
	&= (x \cdot c) \vee (y \cdot c) \\
	&= (x \vee y) \cdot c \\
	&= x \vee y. \qedhere
\end{align*}
\end{proof}

We can now state the sought-after structural characterization result. For any finite member $\alg{A}$ of $\CIdInRL$, either $\alg{A}$ is a Boolean algebra or $\alg{A}$ can be decomposed into two strictly smaller members of $\CIdInRL$ by the decomposition method outlined in this section. Repeated application of this decomposition procedure proves the following theorem.

\begin{theorem}\label{thm:structchar}
Any finite member $\alg{A}$ of $\CIdInRL$ can be constructed using the gluing construction outlined in Section~\ref{sec:gluing} starting from all finite Boolean algebras.
\end{theorem}

As mentioned at the start of this section, the reversal of the gluing construction as outlined cannot be applied to the algebra $\alg{B}^\ast$ from Figure~\ref{fig:locfinite}. However, as will be shown, $\alg{B}^\ast$ can be constructed using the gluing construction from Section~\ref{sec:gluing}, albeit by an infinite number of applications. Characterizing exactly which subclasses of $\CIdInRL$ can be constructed using the gluing construction from Section~\ref{sec:gluing} is left for future work. As the reverse decomposition only depends on the underlying distributive lattice $\alg{A}^+$ being finite, a slight generalization of the theorem above can nonetheless be obtained without further effort.

\begin{corollary}
Any member $\alg{A} \in \CIdInRL$ such that $\alg{A}^+$ is finite can be constructed using the gluing construction in Section~\ref{sec:gluing} starting from all Boolean algebras.
\end{corollary}

Another noteworthy observation is that for the algebras in the preceding corollary the multiplicative order uniquely determines the lattice order and vice versa, hence it suffices to present the simpler multiplicative order. A special case of the gluing construction is the \emph{multiplicative ordinal sum} $\alg{A}\oplus_{\varphi_0}\alg{B}$, where $\varphi_0$ is the unique map from $\{1^A\}$ to $\{0^B\}$. This map is a valid gluing if and only if $0^A\ne 1^A$ in $\alg{A}$ (since the element $a=1^A$ in a gluing must satisfy $a\nleq 0^A$) and $\alg{B}$ is Boolean. For brevity we denote $\oplus_{\varphi_0}$ simply by $\oplus$.

A subdirectly irreducible $\alg{C}\in\CIdInRL$ satisfies $0^A=1^A$ (called \emph{odd} in the context of Sugihara monoids) if and only if $\alg{C}\cong\alg{A}\oplus\alg{1}$ for some $\alg{A}\in\CIdInRL$, where $\alg{1}$ is the one-element algebra. All finite Sugihara monoids can be obtained from the one- and two-element Boolean algebra $\alg{2}$ by the use of (iterated) ordinal sums and direct products (but the variety of Sugihara monoids is not closed under ordinal sums of nonlinear members such as $\alg{2}^2\oplus\alg{1}$). In Figure~\ref{examples} we list multiplicative orders of some small members of $\CIdInRL$, as well as some semilattices that are not the multiplicative order of any $\CIdInRL$.

\begin{figure}
\begin{center}
\begin{tikzpicture}[scale=0.5]
\draw[very thick](0,0)--(-1,1)--(0,2)--(1,1)--(0,0);
\draw( 0,0)node{}--(-1,1)node{}--(0,2)node{}--(0,3)node{}(0,2)--(1,1)node{}--(0,0);
\node at(0,-1)[n]{$\alg{2}^2{\oplus}\alg{1}$};
\end{tikzpicture}
\ 
\begin{tikzpicture}[scale=0.5]
\draw[very thick](0,-1)--(0,-2)(-1,1)--(0,2)--(1,1)--(0,0)--(-1,1);
\draw( 0,0)node{}--(-1,1)node{}--(0,2)node{}--(1,1)node{}--(0,0)--(0,-1)node{}(0,-2)node{};
\node at(0,-3)[n]{$\alg{2}{\oplus}\alg{2}^2$};
\end{tikzpicture}
\quad
\begin{tikzpicture}[scale=0.5]
\draw[very thick](0,0)--(-1,1)--(0,2)--(1,1)--(0,0)
(0,3)--(-1,4)--(0,5)--(1,4)--(0,3);
\draw(0,2)--(1,1)node{}--(0,0)node{}--(-1,1)node{}--(0,2)node{}--(0,3)node{}--(1,4)node{}--(0,5)node{}--(-1,4)node{}--(0,3);
\node at(0,-1)[n]{$\alg{2}^2{\oplus}\alg{2}^2$};
\end{tikzpicture}
\quad
\begin{tikzpicture}[scale=0.5]
\draw[very thick](0,0)--(-1,1)--(0,2)--(1,1)--(0,0)(-2,2)--(-1,3)(2,2)--(1,3)
(0,4)--(0,5);
\draw(0,2)--(1,1)node{}--(0,0)node{}--(-1,1)node{}--(0,2)node{}--(1,3)node{}--(0,4)node{}--(-1,3)node{}--(0,2)(-1,1)--(-2,2)node{}--(-1,3)(1,1)--(2,2)node{}--(1,3)(0,4)--(0,5)node{};
\node at(0,-1)[n]{};
\end{tikzpicture}
\quad
\begin{tikzpicture}[scale=0.5]
\draw[very thick](0,0)--(-1,1)--(0,2)--(1,1)--(0,0)(-2,2)--(-1,3)(2,2)--(1,3)
(0,-1)--(-1,0)--(-1,1)(0,-1)--(1,0)--(1,1)(0,-1)--(0,0)(-1,0)--(0,1)--(1,0)(0,1)--(0,2)(2,2)--(2,1)--(1,2)--(1,3);
\draw(0,2)--(1,1)node{}--(0,0)node{}--(-1,1)node{}--(0,2)node{}--(1,3)node[label=right:$1_y$]{}--(0,4)node[label=right:{$\,0{=}1$}]{}--(-1,3)node[label=left:$1_x\,$]{}--(0,2)(-1,1)--(-2,2)node[label=left:$0_x$]{}--(-1,3)(1,1)--(2,2)node{}--(1,3)(0,-1)node{}(-1,0)node{}(1,0)node{}(2,1)node[label=right:$0_y$]{}--(1,0)(1,2)node{}--(0,1)(0,1)node{};
\node at(0,-2)[n]{$(\alg{2}^2{\oplus}\alg{1}){\times}(\alg{2}\oplus\alg{1})$};
\end{tikzpicture}
\end{center}
\begin{center}
\begin{tikzpicture}[scale=0.5]
\draw[very thick](0,0)--(0,-1)(-1,1)--(0,2)--(1,1);
\draw( 0,0)node{}--(-1,1)node{}--(0,2)node{}--(1,1)node{}--(0,0)(0,-1)node{};
\end{tikzpicture}
\quad
\begin{tikzpicture}[scale=0.5]
\draw[very thick](0,0)--(-1,1)--(0,2)--(1,1)--(0,0)(-1,3)--(0,4)--(1,3);
\draw(0,2)--(1,1)node{}--(0,0)node{}--(-1,1)node{}--(0,2)node{}--(1,3)node{}--(0,4)node{}--(-1,3)node{}--(0,2);
\end{tikzpicture}
\quad
\begin{tikzpicture}[scale=0.5]
\draw[very thick](0,0)--(-1,1)--(0,2)--(1,1)--(0,0)(-2,2)--(-1,3)(0,4)--(1,3);
\draw(0,2)--(1,1)node{}--(0,0)node{}--(-1,1)node{}--(0,2)node{}--(1,3)node{}--(0,4)node{}--(-1,3)node{}--(0,2)(-1,1)--(-2,2)node{}--(-1,3);
\end{tikzpicture}
\quad
\begin{tikzpicture}[scale=0.5]
\draw[very thick](0,0)--(-1,1)--(0,2)--(1,1)--(0,0)(-2,2)--(-1,3)(2,2)--(1,3)
(0,-1)--(-1,0)--(-1,1)(0,-1)--(1,0)--(1,1)(0,-1)--(0,0)(-1,0)--(0,1)--(1,0)(0,1)--(0,2);
\draw(0,2)--(1,1)node{}--(0,0)node{}--(-1,1)node{}--(0,2)node{}--(1,3)node[label=right:$1_y$]{}--(0,4)node[label=right:{$\,0{=}1$}]{}--(-1,3)node[label=left:$1_x\,$]{}--(0,2)(-1,1)--(-2,2)node[label=left:$0_x$]{}--(-1,3)(1,1)--(2,2)node[label=right:$0_y$]{}--(1,3)(0,-1)node{}(-1,0)node{}(1,0)node{}(0,1)node{};
\end{tikzpicture}
\end{center}
\caption{Multiplicative semilattices that can (top) and cannot (bottom) support a $\CIdInRL$}\label{examples}
\end{figure}
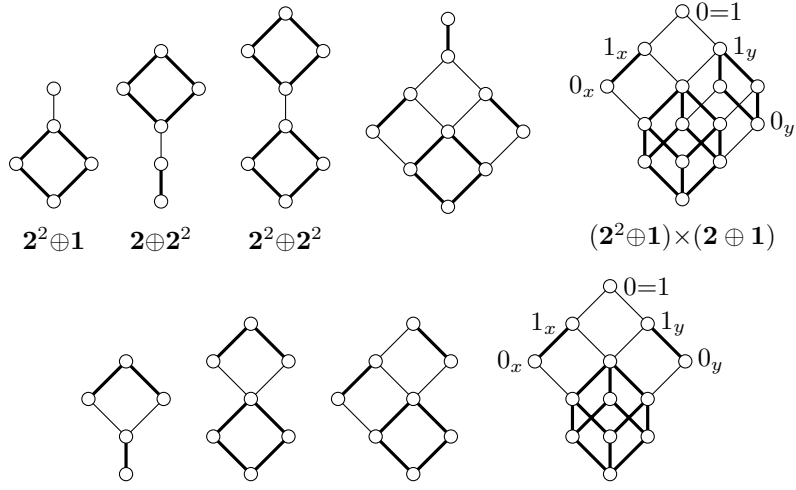

\section{Applications}\label{sec:app}

In this section we discuss two applications of the structural characterization result from Theorem~\ref{thm:structchar}.

\subsection{Distributivity}

Recall from Example~\ref{ex:sugihara} (Fig.\!~\ref{fig:A1}) that for members $\alg{A}$ of $\CIdInRL$, the lattice order does not satisfy the distributive law in general. In this section we apply the structural characterization result from Theorem~\ref{thm:structchar} to show that for each finite $\alg{A} \in \CIdInRL$, the monoidal semilattice $\langle A, \mleq \rangle$ is distributive. Note that for such a finite $\alg{A}$, $\langle A, \mleq \rangle$ is a lattice. But since we have no elegant definition of the join, we work with the notion of distributivity of a semilattice. We say that the semilattice $\langle A, \mleq \rangle$ is \emph{distributive} if for all $x,y,z \in A$,
\[
x \cdot y \mleq z \ \Longrightarrow \ \text{there exists } x', y' \in A \text{ such that } x \mleq x', y \mleq y', \text{ and } z = x' \cdot y'.
\]
Note that a lattice is distributive in the usual sense exactly when it is distributive as a semilattice in this sense (see e.g.~\cite{CHK07}).

\begin{lemma}
Let $\alg{A}, \alg{B} \in \CIdInRL$ with elements $a \in A^-$, $b \in B^-$ and function $\varphi$ such that their gluing $\alg{A} \oplus_\varphi \alg{B}$ is well-defined. If $\langle A, \mleq^A \rangle$ and $\langle B, \mleq^B \rangle$ are distributive, then so is the monoidal semilattice $\langle A \cup B, \mleq \rangle$ of $\alg{A} \oplus_\varphi \alg{B}$.
\end{lemma}
\begin{proof}
	Suppose that $\langle A, \mleq^A \rangle$ and $\langle B, \mleq^B \rangle$ are distributive. To show that also $\langle A \cup B, \mleq \rangle$ is distributive, we consider any $x,y,z \in A \cup B$ such that $x \cdot y \mleq z$. We consider a number of cases. If $x \cdot y \in B$, then $x,y,z \in B$ and the required property follows since $\langle B, \mleq^B \rangle$ is distributive. So suppose that $x \cdot y \in A$. We consider the cases when $x \in A$, $y \in B$, in which case $x \cdot y \mleq z$ means that $x \cdot \varphi^{-1}(y \cdot b) \mleq z$. The other cases are easier or similar. 
	
	Suppose that $z \in A$. Then $x \cdot y \mleq z$ means that $x \cdot \varphi^{-1}(y \cdot b) \mleq z$. Since $\langle A, \mleq \rangle$ is distributive, we get $x', y' \in A$ such that $x \mleq x'$, $\varphi^{-1}(y \cdot b) \mleq y'$ and $x' \cdot y' = z$. Note that $a \mleq \varphi^{-1}(y \cdot b) \mleq y'$, so we apply the fact that $\langle B, \mleq \rangle$ is distributive to $y \cdot b \mleq \varphi(y')$ to get $b', y'' \in B$ such that $b \mleq b'$, $y \mleq y''$, and $\varphi(y') = b' \cdot y ''$. By definition of the range of $\varphi$, $\varphi(y') \mleq b$, hence it follows that $\varphi(y') = b \cdot \varphi(y') = b \cdot b' \cdot y'' = b \cdot y''$. We then have $x \mleq x'$, $y \mleq y''$ and
	\[
	x' \cdot y'' = x' \cdot \varphi^{-1}(y'' \cdot b) = x' \cdot \varphi^{-1}(\varphi(y')) = x' \cdot y' = z.
	\]
	Now suppose that $z \in B$. Then $x \cdot y \mleq z$ means that $x \cdot \varphi^{-1}(y \cdot b) \mleq \varphi^{-1}(z \cdot b)$. Using distributivity of $\langle A, \mleq \rangle$ we obtain $x', y' \in A$ such that $x \mleq x'$, $\varphi^{-1}(y \cdot b) \mleq y'$, and $x' \cdot y' = \varphi^{-1}(z \cdot b)$. Note that $a \mleq \varphi^{-1}(z \cdot b) = x' \cdot y' \mleq x'$, so $\varphi(x')$ is well-defined. Then $y \cdot b \mleq \varphi(y')$ and so $\varphi(x') \cdot y = \varphi(x') \cdot b \cdot y \mleq \varphi(x') \cdot \varphi(y')$. Moreover,
	\[
	\varphi(x') \cdot \varphi(y') = \varphi(x' \cdot y') = \varphi(\varphi^{-1}(z \cdot b)) = z \cdot b \mleq b
	\]
	and hence $\varphi(x') \cdot y \mleq z$. Distributivity of $\langle B, \mleq \rangle$ then gives $x'', y'' \in B$ such that $x \mleq x' \mleq \varphi(x') \mleq x''$, $y \mleq y''$, and $z = x'' \cdot y''$.
\end{proof}

The next result now follows from Theorem~\ref{thm:structchar} and the fact that any Boolean algebra is distributive. We conjecture that this result holds for any member of $\CIdInRL$, not only the finite ones.

\begin{theorem}
For any finite $\alg{A} \in \CIdInRL$, $\langle A, \mleq \rangle$ is a distributive semilattice.
\end{theorem}

\subsection{Locally Finiteness}
For another application of Theorem~\ref{thm:structchar}, in this section we construct a sequence of $1$-generated members of $\CIdInRL$ with increasing cardinality, showing that the variety $\CIdInRL$ is not locally finite. This is in contrast with two well-known subvarieties of $\CIdInRL$, namely Boolean algebras and Sugihara monoids
that are both locally finite varieties (see \cite[Theorem 1]{R07} for the latter case).

For every $i \in \mathbb{N}$, we define $\alg{B}_i$ to be the four-element Boolean algebra with universe $\{ 0_i, x_i, \negmin x_i, 1_i \}$. Given two such (disjoint) algebras $\alg{B}_i$ and $\alg{B}_{i+1}$, we define two types of gluing: one that glues $\alg{B}_{i+1}$ ``on the left of $\alg{B}_i$'' if $i$ is even, and one that glues $\alg{B}_{i+1}$ ``on the right of $\alg{B}_i$'' if $i$ is odd. That is, 
\begin{align*}
\text{ for even $i$, }\varphi_i & \text{ is defined by } 1_i \mapsto \negmin x_{i+1} \text{ and } x_i \mapsto 0_{i+1}; \\
\text{ for odd $i$, }\varphi_i & \text{ is defined by } 1_i \mapsto x_{i+1} \text{ and } \negmin x_i\mapsto 0_{i+1}.
\end{align*}
Moreover, let $\alg{2}$ denote the 2-element Boolean algebra, with universe $\{ 0,1 \}$, and let $\alg{B}_i \oplus_{\psi_i} \alg{2}$ be the gluing given by $\psi_i \colon \alg{B}_i \to \alg{2}$ where $\psi_i(1_i) = 0$. For example, we can now express $\alg{A}_1$ from Figure~\ref{fig:A1} as $(\alg{B}_0 \oplus_{\varphi_0} \alg{B}_1) \oplus_\psi \alg{2}$, where $a = x_0$ and $b = x_1$.
In general, define the algebras, depicted in Figure~\ref{fig:locfinite}:
\[
\alg{A}_n:= ((\ldots ((\alg{B}_0 \oplus_{\varphi_0} \alg{B}_1) \oplus_{\varphi_1} \alg{B}_2 ) \oplus_{\varphi_2} 
\ldots  \oplus_{\varphi_{n-2}} \alg{B}_{n-1} ) \oplus_{\varphi_{n-1}} \alg{B}_{n} ) 
\oplus_{\psi_n} \mathbf{2}.
\]

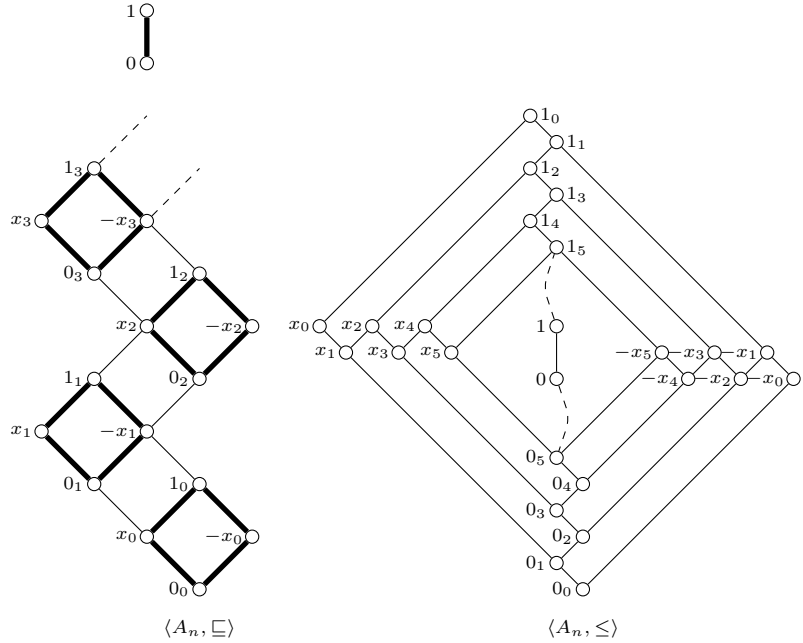
\begin{figure}[ht]
\begin{center}
	\scriptsize

\begin{tikzpicture}[baseline=0pt,scale=0.7]
	\node(1) at (-0.5,6.5)[label=left:$1$]{};
	\node(0) at (-0.5,5.5)[label=left:$0$]{};
	\draw(0)  edge[boolean] (1);
	
	\node(1x3) at (-1.5,3.5)[label=left:$1_3$]{};
	\node(x3) at (-2.5,2.5)[label=left:$x_3$]{};
	\node(nx3) at (-0.5,2.5)[label=left:$\negmin x_3$]{};
	\node(0x3) at (-1.5,1.5)[label=left:$0_3$]{};
	\draw(1x3) edge[boolean] (nx3);
	\draw(1x3) edge[boolean] (x3);
	\draw(0x3) edge[boolean] (x3);
	\draw(nx3) edge[boolean] (0x3);
	\draw[dashed](1x3) -- (-0.5,4.5);
	\draw[dashed](nx3) -- (0.5,3.5);
	
	\node(1x2) at (0.5,1.5)[label=left:$1_2$]{};
	\node(x2) at (-0.5,0.5)[label=left:$x_2$]{};
	\node(nx2) at (1.5,0.5)[label=left:$\negmin x_2$]{};
	\node(0x2) at (0.5,-0.5)[label=left:$0_2$]{};
	\draw(1x2) edge[boolean] (nx2);
	\draw(1x2) edge[boolean] (x2);
	\draw(0x2) edge[boolean] (x2);
	\draw(nx2) edge[boolean] (0x2);
	\draw(1x2) edge (nx3);
	\draw(x2) edge (0x3);
	
	\node(1x1) at (-1.5,-0.5)[label=left:$1_1$]{};
	\node(x1) at (-2.5,-1.5)[label=left:$x_1$]{};
	\node(nx1) at (-0.5,-1.5)[label=left:$\negmin x_1$]{};
	\node(0x1) at (-1.5,-2.5)[label=left:$0_1$]{};
	\draw(1x1) edge[boolean] (nx1);
	\draw(1x1) edge[boolean] (x1);
	\draw(0x1) edge[boolean] (x1);
	\draw(nx1) edge[boolean] (0x1);
	\draw(1x1) edge (x2);
	\draw(nx1) edge (0x2);
	
	\node(1x0) at (0.5,-2.5)[label=left:$1_0$]{};
	\node(x0) at (-0.5,-3.5)[label=left:$x_0$]{};
	\node(nx0) at (1.5,-3.5)[label=left:$\negmin x_0$]{};
	\node(0x0) at (0.5,-4.5)[label=left:$0_0$]{};
	\draw(1x0) edge[boolean] (nx0);
	\draw(1x0) edge[boolean] (x0);
	\draw(0x0) edge[boolean] (x0);
	\draw(nx0) edge[boolean] (0x0);
	\draw(1x0) edge (nx1);
	\draw(x0) edge (0x1);
	
	\node at (0.5,-5.25)[n]{$\langle A_n, \mleq \rangle$};
\end{tikzpicture} \quad
\begin{tikzpicture}[baseline=0pt,scale=0.7]
\node(1) at (0,0.5)[label=left:$1$]{};
\node(0) at (0,-0.5)[label=left:$0$]{};
\draw(0)  edge (1);

\node(1x5) at (0,2)[label=right:$1_5$]{};
\node(0x5) at (0,-2)[label=left:$0_5$]{};
\node(nx5) at (2,0)[label=left:$\negmin x_5$]{};
\node(x5) at (-2,0)[label=left:$x_5$]{};
\draw(1x5)  edge (x5);
\draw(0x5)  edge (nx5);
\draw(x5)  edge (0x5);
\draw(1x5)  edge (nx5);

\draw[dashed, rounded corners](1) -- (-0.25,1.25) -- (1x5);
\draw[dashed,rounded corners](0)  -- (0.25,-1.25) --(0x5);

\node(1x4) at (-0.5,2.5)[label=right:$1_4$]{};
\node(0x4) at (0.5,-2.5)[label=left:$0_4$]{};
\node(x4) at (-2.5,0.5)[label=left:$x_4$]{};
\node(nx4) at (2.5,-0.5)[label=left:$\negmin x_4$]{};
\draw(1x4) edge (1x5);
\draw(x4) edge (x5);
\draw(1x4) edge (x4);
\draw(0x4) edge (0x5);
\draw(nx4) edge (nx5);
\draw(nx4) edge (0x4);

\node(1x3) at (0,3)[label=right:$1_3$]{};
\node(0x3) at (0,-3)[label=left:$0_3$]{};
\node(x3) at (-3,0)[label=left:$x_3$]{};
\node(nx3) at (3,0)[label=left:$\negmin x_3$]{};
\draw(1x3) edge (1x4);
\draw(x3) edge (0x3);
\draw(0x3) edge (0x4);
\draw(nx3) edge (nx4);
\draw(nx3) edge (1x3);
\draw(x3) edge (x4);

\node(1x2) at (-0.5,3.5)[label=right:$1_2$]{};
\node(0x2) at (0.5,-3.5)[label=left:$0_2$]{};
\node(x2) at (-3.5,0.5)[label=left:$x_2$]{};
\node(nx2) at (3.5,-0.5)[label=left:$\negmin x_2$]{};
\draw(1x2) edge (1x3);
\draw(x2) edge (x3);
\draw(1x2) edge (x2);
\draw(0x2) edge (0x3);
\draw(nx2) edge (nx3);
\draw(nx2) edge (0x2);

\node(1x1) at (0,4)[label=right:$1_1$]{};
\node(0x1) at (0,-4)[label=left:$0_1$]{};
\node(x1) at (-4,0)[label=left:$x_1$]{};
\node(nx1) at (4,0)[label=left:$\negmin x_1$]{};
\draw(1x1) edge (1x2);
\draw(x1) edge (0x1);
\draw(0x1) edge (0x2);
\draw(nx1) edge (nx2);
\draw(nx1) edge (1x1);
\draw(x1) edge (x2);

\node(1x0) at (-0.5,4.5)[label=right:$1_0$]{};
\node(0x0) at (0.5,-4.5)[label=left:$0_0$]{};
\node(x0) at (-4.5,0.5)[label=left:$x_0$]{};
\node(nx0) at (4.5,-0.5)[label=left:$\negmin x_0$]{};
\draw(1x0) edge (1x1);
\draw(x0) edge (x1);
\draw(1x0) edge (x0);
\draw(0x0) edge (0x1);
\draw(nx0) edge (nx1);
\draw(nx0) edge (0x0);

\node at (0.5,-5.25)[n]{$\langle A_n, \leq \rangle$};

\end{tikzpicture}
\end{center}
\caption{Infinite sequence of 1-generated algebras $\alg{A}_n$}\label{fig:locfinite}
\end{figure}

\noindent Note that by direct computation, we obtain
\begin{align*}
1_j &= -x_{j-1}\vee 1 & 0_j &= x_{j-1}\wedge 1 \\
x_{j} &= x_{j-1}\wedge 1_j &  -x_j &= -x_{j-1}\vee 0_j
\end{align*}
for any $j\geq 1$, and so $\alg{A}_n$ is generated by the single element $x_0$. Moreover, it follows by iterated application of Theorem~\ref{thm:gluing} that $\alg{A}_n$ indeed belongs to $\CIdInRL$.

\begin{proposition}
The variety $\CIdInRL$ is not locally finite.
\end{proposition}

\noindent \textbf{Acknowledgements}

\noindent Some of the computations leading to our results have been obtained with the help of Prover9 and Mace4~\cite{P9M4}. 
The authors acknowledge the support of funding from the European Union’s Horizon 2020 research and innovation programme 
under the Marie Sk{\l}odowska-Curie grant agreement No 689176. 



\end{document}